\documentclass[10pt]{amsart}
\usepackage{amsrefs}
\usepackage{amssymb}
\usepackage[all]{xy}
\usepackage{multicol}
\usepackage{hyperref}

\usepackage[draft]{changes} 

\usepackage{tikz}
\usetikzlibrary{matrix,arrows}

\DeclareMathOperator{\cd}{cd}

\DeclareMathOperator{\gr}{\mathrm{gr}}

\DeclareMathOperator{\Img}{Im}

\DeclareMathOperator{\invlim}{\varprojlim}
\DeclareMathOperator{\Ker}{Ker}

\DeclareMathOperator{\ab}{ab}
\DeclareMathOperator{\Zen}{Z}
\DeclareMathOperator{\ext}{Ext}

\DeclareFontFamily{U}{wncy}{}
\DeclareFontShape{U}{wncy}{m}{n}{<->wncyr10}{}
\DeclareSymbolFont{mcy}{U}{wncy}{m}{n}
\DeclareMathSymbol{\Sha}{\mathord}{mcy}{"58}
\DeclareMathSymbol{\sha}{\mathord}{mcy}{"78}

\begin{document}

	\newtheorem{thm}{Theorem}[section]
	\newtheorem{cor}[thm]{Corollary}
	\newtheorem{lem}[thm]{Lemma}
	\newtheorem{fact}[thm]{Fact}
	\newtheorem{prop}[thm]{Proposition}
	\newtheorem{defin}[thm]{Definition}
	\newtheorem{exam}[thm]{Example}
	\newtheorem{examples}[thm]{Examples}
	\newtheorem{rem}[thm]{Remark}
	\newtheorem{case}{\sl Case}
	\newtheorem{claim}{Claim}
	\newtheorem{question}[thm]{Question}
	\newtheorem{conj}[thm]{Conjecture}
	\newtheorem*{notation}{Notation}
	\swapnumbers
	\newtheorem{rems}[thm]{Remarks}
	\newtheorem*{acknowledgment}{Acknowledgments}
	\newtheorem*{thmno}{Theorem}
	
	\newtheorem{questions}[thm]{Questions}
	\numberwithin{equation}{section}

	\newcommand{\inv}{^{-1}}
	\newcommand{\isom}{\cong}
	\newcommand{\dbC}{\mathbb{C}}
	\newcommand{\F}{\mathbb{F}}
	\newcommand{\dbN}{\mathbb{N}}
	\newcommand{\Q}{\mathbb{Q}}
	\newcommand{\dbR}{\mathbb{R}}
	\newcommand{\dbU}{\mathbb{U}}
	\newcommand{\Z}{\mathbb{Z}}
	\newcommand{\calG}{\mathcal{G}}
	\newcommand{\calX}{\mathcal{X}}
	\newcommand{\calY}{\mathcal{Y}}
	\newcommand{\K}{\mathbb{K}}
	\newcommand{\rmH}{\mathrm{H}}
	\newcommand{\bfH}{\mathbf{H}}
	\newcommand{\rmr}{\mathrm{r}}
	\newcommand{\Span}{\mathrm{Span}}
	\newcommand{\eue}{\mathbf{e}}
	\newcommand{\bfLam}{\mathbf{\Lambda}}
	\newcommand{\dbl}{[\![}
	\newcommand{\bdr}{]\!]}
	\newcommand{\FppG}{\F_p[\![G]\!]}
	\newcommand{\grFp}[1]{\gr\F_p[\![{#1}]\!]}
	\newcommand{\set}[2]{\left\lbrace{#1}\ \big\vert\ {#2}\right\rbrace}

    \newcommand{\gen}[1]{\left\langle{#1}\right\rangle}
	
	\newcommand{\pres}[2]{\left\langle\:{#1}\:\mid\: {#2}\:\right\rangle}
	
	
	\newcommand{\hac}{\hat c}
	\newcommand{\hatheta}{\hat\theta}
	
	
	\title[Demu\v skin variations and absolute Galois Groups]
	{ Variations of Demu\v skin groups\\ that are not absolute Galois groups 
}

\author{Simone Blumer}

\email{simoneblumer96@gmail.com}
\author{Claudio Quadrelli}
\address{Department of Science \& High-Tech, University of Insubria, Como, Italy EU}
\email{claudio.quadrelli@uninsubria.it}
\date{\today}
\dedicatory{Dedicated to John P. Labute, \\ 61 years after his PhD and 20 years after his paper on mild pro-$p$ groups}

\begin{abstract}
	We construct two families of examples of pro-$p$ groups, with rather elementary presentations, that do not complete into 1-cyclotomic oriented pro-$p$ groups.
	These provide brand new examples of pro-$p$ groups that do not occur as maximal pro-$p$ Galois groups of fields containing a root of unity of order $p$ --- and thus, as absolute Galois groups.
	Moreover, we show that these pro-$p$ groups may not be ruled out as maximal pro-$p$ Galois groups employing other cohomological properties that are known to hold for all maximal pro-$p$ Galois groups, such as the triple Massey vanishing property, or the quadraticity of $\F_p$-cohomology.
\end{abstract}



\maketitle

\section{Introduction}
\label{sec:intro}

\subsection{Framework}
Let $p$ be a prime number, and, given a field $\K$, let $\K(p)$ denote its {\sl $p$-closure}, i.e., $\K(p)$ is the compositum of all finite Galois $p$-extensions of $\K$.
The {\sl maximal pro-$p$ Galois group} of the field $\K$, denoted by $G_{\K}(p)$, is the Galois group of the Galois extension $\K(p)/\K$: it is a pro-$p$ group and it is the maximal pro-$p$ quotient of the {\sl absolute Galois group} of $\K$.

Detecting which pro-$p$ groups occur as maximal pro-$p$ Galois groups is one of the major pursuits in current research in Galois theory (see, e.g., \cite[\S~25.16]{friedjarden} and \cite[\S~2.2]{birs}).
Already the production of new concrete examples of pro-$p$ groups which do not occur as maximal pro-$p$ Galois groups is considered a remarkable achievement (see, e.g., \cites{BLMS, rogel, cq:chase, BQW}).
In recent years --- especially after the proof of the Norm Residue Theorem --- the research on maximal pro-$p$ Galois groups focused on the investigation of their cohomology, and of associated graded algebras, and this led to the discovery of new obstructions for the realization of pro-$p$ groups as maximal pro-$p$ Galois groups, and to the formulation of new conjectures --- see, e.g., \cites{posi,cq:bk,EM:Massey, MPQT,mt:Massey,HW} and references therein.

One of the cohomological properties that provided new obstructions for the realization of pro-$p$ groups and maximal pro-$p$ Galois groups is {\sl 1-cyclotomicity}.
A pro-$p$ group is said to be 1-cyclotomic if there exists a continuous $G$-module $M$, $M\simeq \Z_p$ as an abelian pro-$p$ group, such that for every closed subgroup $H$ of $G$, and for every $n\geq1$, the natural map
\begin{equation}\label{eq:1cyc intro}
	\rmH^1(H,M_H/p^nM_H)\longrightarrow\rmH^1(H,M_H/pM_H),
\end{equation}
--- here $M_H$ denotes the restriction of $M$ as an $H$-module ---,
induced by the epimorphism of $H$-modules
$M_H/p^nM_H\to M_H/pM_H$, is surjective.
By Hilbert~90, maximal pro-$p$ Galois groups of fields containing a root of unity of order $p$ are 1-cyclotomic, as observed first by J.~Labute (see Section~\ref{ssec:GK 1cyc} below).
In the last years this property has been put under the spotlight by several authors, see, e.g., \cite{eq:kummer,qw:cyc,IlirSlobo,cq:chase,DCF,MerScaH90}.

\subsection{Main results: new examples}
The goal of this work is to produce two families of pro-$p$ groups, $\mathcal{F}_1$ and $\mathcal{F}_2$, that are ``variations'' of {\sl Demu\v skin groups} --- namely, they have very similar presentations ---, which are not 1-cyclotomic.
Demu\v skin groups are Poincar\'e duality pro-$p$ groups of dimension $2$ --- hence, they have a single defining relation: for example, if $p\neq2$ a pro-$p$ group $G$ is a Demu\v skin group if and only if it has a minimal pro-$p$ presentation 
\begin{equation}\label{eq:pres demushkin}
	G=\left\langle\:x_1,y_1,\ldots,x_d,y_d\:\mid\: x_1^q[x_1,y_1]\cdots[x_{d},y_d]=1 \:\right\rangle_{\hat p}
\end{equation}
for some $p$-power $q=p^k$ with $k\in\{1,2,\ldots,\infty\}$ (here we adopt the convention $p^\infty=0$), and $d\geq1$, see, e.g., \cite[Ch.~I, \S~4.5]{serre:galc}.
Demu\v skin groups were introduced by S.P.~Demu\v skin (see \cite{demushkin0,demushkin1,demushkin2}): these pro-$p$ groups show up in Galois theory as the maximal pro-$p$ Galois group of a $p$-adic local field containing a root of unity of order $p$ is a Demu\v skin group (see, e.g., \cite[Ch.~II, \S~5.6]{serre:galc}).
Moreover, {\sl all} Demu\v skin groups --- and their cohomology, and their associated graded algebras --- satisfy all properties that are known, or conjectured, to hold for maximal pro-$p$ Galois groups (which we will recall in \S~\ref{ssec:galois-like} below). In particular, all Demu\v skin groups are 1-cyclotomic.
Nevertheless, it is still not known whether {\sl any} Demu\v skin group occurs as the maximal pro-$p$ group of a field (see, e.g., \cite[\S~11.2]{etc_new}).

The family $\mathcal{F}_1$ consists of pro-$p$ groups with minimal pro-$p$ presentation
\begin{equation}\label{eq:presG2}
	G=\left\langle\:x_1,y_1,\ldots,x_d,y_d\:\mid\: [x_1^q,y_1]\cdots[x_{d},y_d]=1 \:\right\rangle_{\hat p}
\end{equation}
for some $d\geq2$ and some $p$-power $q=p^k$ with $k\in\{1,2,\ldots\}$, and $k\geq2$ if $p=2$ --- hence $q\neq0$. We prove the following.

\begin{thm}\label{thmG2}
	Every pro-$p$ group $G$ in the family $\mathcal{F}_1$ is not 1-cyclotomic.
	In particular, $G$ does not occur as the maximal pro-$p$ Galois group of a field containing a root of unity of order $p$ {\rm(}and also $\sqrt{-1}$ if $p=2${\rm)}.
\end{thm}

The family $\mathcal{F}_2$ consists of pro-$p$ groups with minimal pro-$p$ presentation
\begin{equation}\label{eq:presG1}
	G=\left\langle\:x_1,y_1,\ldots,x_d,y_d\:\mid\: x_1^q[x_1,y_1]\cdots[x_{d},y_d]=[z_1,z_2]=1 \:\right\rangle_{\hat p}
\end{equation}
for some $p$-power $q=p^k$ with $k\in\{1,2,\ldots,\infty\}$ (and $k\geq2$ if $p=2$), $d\geq2$, and $z_1,z_2\in\{x_1,\ldots,y_d\}$ such that $z_1\neq z_2$ and $\{z_1,z_2\}\neq\{x_i,y_i\}$ for any $1\leq i\leq d$.
In other words, a pro-$p$ group $G$ in $\mathcal{F}_2$ is the quotient of a Demu\v skin group with minimal presentation~\eqref{eq:pres demushkin} over the commutator of two distinct generators that are not paired in a commutator of the Demu\v skin defining relation.

\begin{thm}\label{thmA1}
	Every pro-$p$ group $G$ in the family $\mathcal{F}_2$ is not 1-cyclotomic.
	In particular, $G$ does not occur as the maximal pro-$p$ Galois group of a field containing a root of unity of order $p$ {\rm(}and also $\sqrt{-1}$ if $p=2${\rm)}.
\end{thm}
 
Thus, the pro-$p$ groups in the families $\mathcal{F}_1$ and $\mathcal{F}_2$ are genuine new examples of pro-$p$ groups that do not --- but are very ``close'' to --- occur as maximal pro-$p$ Galois groups.
We underline that, although Demu\v skin groups are some of the (extremely few) examples of pro-$p$ groups which are known to be 1-cyclotomic, the presence of a further, elementary and apparently harmless ``variation'' (in particular in $\mathcal{F}_2$) is enough to lose 1-cyclotomicity.

\subsection{Galois-like properties}\label{ssec:galois-like}
Next, we study the pro-$p$ groups in the families $\mathcal{F}_1$ and $\mathcal{F}_2$, focusing on several important theorems and conjectures on maximal pro-$p$ Galois groups: the {\sl Norm Residue Theorem}, the {\sl Massey Vanishing Conjecture}, and the {\sl Koszulity Conjectures}, which we briefly recall below.

\smallskip

 \noindent{\bf (1)} The {\sl Norm Residue Theorem} (proved by M.~Rost and V.~Voevodsky, with a ``patch'' by Ch.~Weibel, for an overview see, e.g., \cite{HW:book}) implies that the $\F_p$-cohomology algebra 
    $$\bfH^\bullet(G_{\K}(p),\F_p)=\bigsqcup_{n\geq0}\rmH^n(G_{\K}(p),\F_p),$$ with $\K$ a field containing a root of unity of order $p$, endowed with the cup-product, is a {\sl quadratic} $\F_p$-algebra (see Section~\ref{ssec:H} below): i.e., it is generated in degree 1, and its defining relations are generated in degree 2 (see, e.g., \cite{qsv:quadratic} and references therein).

\smallskip
\noindent{} {\bf (2)}    For $n\geq 2$ and $G$ a pro-$p$ group, the {\sl $n$-fold Massey product} is a multi-valued product which associates a sequence of elements $\alpha_1,\ldots,\alpha_n$ of $\rmH^1(G,\F_p)$ to a subset 
    $$\langle\alpha_1,\ldots,\alpha_n\rangle\subseteq\rmH^2(G,\F_p)$$ (see Section~\ref{ssec:Massey} below).
    E.~Matzri proved that if $G=G_{\K}(p)$ with $\K$ a field containing a root of unity of order $p$, then every 3-fold Massey product contains 0 (see \cites{eli:Massey,EM:Massey}), and it was conjectured by J.~Mina\v c and N.D.~T\^an that every $n$-fold Massey product contains 0 for $n\geq3$ (see \cite{mt:conj}). 
    
\noindent  For an overview on Massey products in Galois cohomology see \cite{mt:Massey}.
    
\smallskip
    \noindent{\bf (3)} {\sl Koszul algebras} are a particular kind of quadratic $\F_p$-algebras: a graded $\F_p$-algebra $A=\bigoplus_{i\geq 0}A_i$, with $A_0=\F_p$, is Koszul if the trivial $A$-module $M=\F_p$ admits a linear free resolution, that is an exact sequence \begin{equation}\label{eq:linearRes}
    \dots\to F_1\to F_0\to M\to 0,\end{equation} where $F_i$ is a free $A$-module generated by elements of degree $i$. Koszul algebras boast a very nice behavior --- e.g., a ``small'' projective resolution may easily be computed, and they are completely determined by their $\ext$-algebra $\ext_A^\bullet(\F_p,\F_p)$.

\noindent
    It was conjectured by L.~Positselski that the $\F_p$-cohomology algebra $\bfH^\bullet(G_{\K}(p),\F_p)$, with $\K$ a field containing a root of unity of order $p$, is Koszul; and later, it was conjectured by J.Mina\v c et al. that such an algebra satisfies even a stronger condition, called {\sl universal Koszulity}. An $\F_p$-algebra $A$ is universally Koszul if there exists a linear free resolution as in (\ref{eq:linearRes}), where $M$ is any quotient $A/I$, where $I$ is a (left) ideal of $A$ generated in degree $1$.

\noindent    
    Moreover, a strengthening of Positselki's conjecture was proposed by Th.~Weigel \cite{thomas:koszul} who predicted that also the graded group $\F_p$-algebra 
    $$\mathrm{gr}\,\F_p[\![G_{\K}(p)]\!]=\coprod_{n\geq0}I_G^n/I_G^{n+1},$$ induced by the augmentation ideal $I_G\lhd\F_p[\![G_{\K}(p)]\!]$ (see Section~\ref{sec:gradedAlgebras} below), is quadratic, Koszul, and it is the {\sl quadratic dual} of $\bfH^\bullet(G_{\K}(p),\F_p)$.

\noindent
For an overview of Koszulity in Galois cohomology see \cites{posi,MPQT,expo,MPPT}.

\smallskip

We prove that the $\F_p$-cohomology of the pro-$p$ groups in the families $\mathcal{F}_1,\mathcal{F}_2$ matches the aforementioned theorems and conjectures.

\begin{prop}\label{prop:properties intro}
	Let $G$ be a pro-$p$ group lying in the family $\mathcal{F}_1$ or in the family $\mathcal{F}_2$.
	\begin{itemize}
		\item[(1)] The $\F_p$-cohomology algebra $\bfH^\bullet(G,\F_p)$ is quadratic.
		\item[(2.a)] If $G\in\mathcal{F}_1$ and $n\leq q$, then $G$ satisfies a strong variant of the $n$-fold Massey vanishing property.
        \item[(2.b)] If $G\in\mathcal{F}_2$, then $G$ satisfies the $3$-fold and the 4-fold Massey vanishing properties.
		\item[(3)] The algebras $\bfH^\bullet(G,\F_p)$ and $\mathrm{gr}\,\FppG$ are Koszul, and they are quadratic dual to each other. Moreover, the former is universally Koszul.
	\end{itemize}
\end{prop}

A key point for the proof of Proposition~\ref{prop:properties intro}--(3) is the fact that the pro-$p$ group in $\mathcal{F}_1,\mathcal{F}_2$ are {\sl mild}.
Mild pro-$p$ groups were introduced by J.P.~Labute in \cite{labute:mild} to study the Galois groups of pro-$p$ extensions of number fields with restricted ramification --- for an overview on mild pro-$p$ groups see also \cite[\S~2--3]{Jochen}, \cite[\S~2.5]{qsv:quadratic} and \cite[\S~4]{expo}.

Altogether, the pro-$p$ groups in the families $\mathcal{F}_1,\mathcal{F}_2$ ``behave'' pretty much like maximal pro-$p$ Galois groups, apart from 1-cyclotomicity.
This suggests that 1-cyclotomicity provides a rather refined and powerful tool to study maximal pro-$p$ Galois groups, and it succeeds in detecting pro-$p$ groups which do not occur as maximal pro-$p$ Galois groups when other methods fail.

\subsection{Graded restricted Lie algebras}\label{ssec:intro group algebras}

Finally, we focus on the graded restricted Lie algebras $\mathrm{gr}\:G$, with $G\in\mathcal{F}_1\cup\mathcal{F}_2$, induced by the $p$-Zassenhaus filtration of $G$, which may be seen as ``linearizations'' of the pro-$p$ groups $G$.
On the other hand, they are interesting objects on their own.
Since the pro-$p$ groups in $\mathcal{F}_1,\mathcal{F}_2$ are mild, the associated graded restricted Lie algebras may be described easily.
Explicitly, for $G$ in $\mathcal{F}_1$, respectively $G$ in $\mathcal{F}_2$, $\mathrm{gr}\:G$ has presentation
\[
\mathrm{gr}\,G =\begin{cases}
    \langle\:X_1,Y_1,\ldots,X_d,Y_d\:\mid\:[X_2,Y_2]+\ldots+[X_d,Y_d]=0\:\rangle,\\
    \langle\:X_1,Y_1,\ldots,X_d,Y_d\:\mid\:[X_1,Y_1]+\ldots+[X_d,Y_d]=[Z_1,Z_2]=0\:\rangle
\end{cases}\]
respectively --- in the latter $Z_1,Z_2\in\{X_1,\ldots,Y_d\}$ and $\{Z_1,Z_2\}\neq\{X_i,Y_i\}$ for any $1\leq i\leq d$, in accordance with $z_1,z_2\in G$.

The notion of Koszul algebra may be transposed to the setting of graded restricted Lie algebras (see \cite[\S 1.3]{simone:koszul} and \S~\ref{sec:gradedAlgebras} below).
For $G$ in $\mathcal{F}_1$ or in $\mathcal{F}_2$,
the associated restricted Lie algebra $\mathrm{gr}\:G$ satisfies a stronger version of Koszulity, as stated by the following.

\begin{prop}\label{prop:gradedalgebras intro}
Let $G$ be a pro-$p$ group lying in the family $\mathcal{F}_1$ or in the family $\mathcal{F}_2$. 
If $p\neq 2$ the restricted Lie algebra $\mathrm{gr}\,G$ is Bloch-Kato, namely, every subalgebra of $\mathrm{gr}\,G$ generated in degree 1 is Koszul.
\end{prop}

Proposition~\ref{prop:properties intro}--(1) raises a question inspired also by Proposition~\ref{prop:gradedalgebras intro}.
A pro-$p$ group is said to be {\sl Bloch-Kato} if the $\F_p$-cohomology of every closed subgroup is quadratic (see \cite{cq:bk} and Section~\ref{ssec:H} below) --- so, Bloch-Kato-ness is a ``hereditary'' property, like 1-cyclotomicity.
Proposition~\ref{prop:properties intro}--(1) states that $\bfH^\bullet(G,\F_p)$ is quadratic for $G$ in $\mathcal{F}_1\cup \mathcal{F}_2$, but it says nothing about the subgroups of $G$.
Hence, similarly to Proposition \ref{prop:gradedalgebras intro}, it is natural to ask the following.

\begin{question}\label{ques:BK}
	Are the pro-$p$ groups in the families $\mathcal{F}_1,\mathcal{F}_2$ Bloch-Kato pro-$p$ groups?
\end{question}
In order to prove Theorems~\ref{thmG2}--\ref{thmA1} we studied some open subgroups of the pro-$p$ groups in $\mathcal{F}_1,\mathcal{F}_2$, which turned out to have quadratic $\F_p$-cohomology --- see Remarks~\ref{rem:U BK 1}--\ref{rem:G2 U kummer}--\ref{rem:G2 V quad} ---, so answering Question~\ref{ques:BK} will require a more thorough investigation.

Although all known examples indicate a close connection between the Bloch–Kato property of a pro-p group and that of its associated Lie algebra, it remains open whether a pro-p group with quadratic cohomology and Bloch–Kato associated Lie algebra is necessarily Bloch–Kato.
\subsection{An open question}

We conclude with a question on another variation of Demu\v skin groups, which has not been studied in this work.
Consider the pro-$p$ group $G$ with minimal presentation
\begin{equation}\label{eq:pres question}
	G=\langle\:x_1,y_1,\ldots,x_d,y_d\:\mid\:[x_1,y_1]=[x_2,y_2]=\ldots=[x_d,y_d]\:\rangle_{\hat p},
\end{equation}
with $d\geq3$.
Namely, $G$ is the amalgamated free pro-$p$ product 
\[
G=G_1\amalg_{H_1}G_2\amalg_{H_2}\cdots\amalg_{H_{d-2}}G_{d-1}
\]
of $d-1$ isomorphic Demu\v skin groups on four generators
\[
G_i=\langle\:x_i,y_i,x_{i+1},y_{i+1}\:\mid\:[x_i,y_i]=[x_{i+1},y_{i+1}]\:\rangle_{\hat p}
\qquad i=1,\ldots,d-1\]
with amalgams the intersection $H_i=G_i\cap G_{i+1}=\langle x_{i+1},y_{i+1}\rangle$.

\begin{question}
	Is the pro-$p$ group $G$ in \eqref{eq:pres question}
	1-cyclotomic?
\end{question}

We suspect that the above question has negative answer, like for the pro-$p$ groups in $\mathcal{F}_1,\mathcal{F}_2$, and thus it does not occur as the maximal pro-$p$ Galois group of a field containing a root of unity of order $p$.

Last, but not least, we deem the pro-$p$ groups in the families $\mathcal{F}_1$ and $\mathcal{F}_2$ good showcases, useful to display the techniques that may be employed to study, in a concrete way, some cohomological features of pro-$p$ groups relevant from a Galois-theoretic point of view, such as 1-cyclotomicity and Massey products.
Also for this reason, we hope that our work may be appreciated --- in particular --- by students and young researchers who are approaching Galois cohomology of pro-$p$ groups.

{\small \subsection*{Acknowledgments}
	The first-named author gratefully acknowledges Thomas Weigel for his support and for allowing him to continue his research at the University of Milano-Bicocca during a period without formal affiliation.
	The authors acknowledge their membership to the national group GNSAGA (Algebraic Structures and Algebraic Geometry) of the National Institute of Advanced Mathematics -- a.k.a. INdAM -- ``F. Severi''.
}

\section{Preliminaries}\label{sec:preli}

\subsection{Notation}\label{ssec:notation}
Henceforth, every subgroup of a pro-$p$ group will be tacitly assumed to be closed with respect to the pro-$p$ topology, and sets of generators, and presentations, of pro-$p$ groups will be intended in the topological sense.
Given two elements $x,y\in G$ of a pro-$p$ group $G$, we adopt the notation
\[
y^x=x^{-1}\cdot y\cdot x\qquad\text{and}\qquad[x,y]=(y^{-1})^x\cdot y=x^{-1}y^{-1}xy.
\]
Given a positive integer $n$, one has the normal subgroups
\[ G^{p^n} =\left\langle\: x^{p^n}\:\mid\:x\in G\:\right\rangle\qquad\text{and}\qquad
G' =\left\langle\:[x,y]\:\mid\:x,y\in G\:\right\rangle.
\]
More in general, if $H_1,H_2$ are subgroups of $G$, $[H_1,H_2]$ denotes the subgroup of $G$ generated by the elements $[x,y]$ where $x\in H_1$ and $y\in H_2$.
Finally, the Frattini subgroup of $G$ is $\Phi(G)=G^p\cdot G'$.

Later on, we will make use of the following fact.

\begin{fact}\label{fact}
	Let $G$ be a finitely generated pro-$p$ group, and let $\calX$ be a subgroup of $G$.
	Then $\calX$ is a minimal generating set of $G$ if and only if $\bar\calX=\{xG'\mid x\in\calX\}$ is a minimal generating set of the abelianization $G^{\ab}$ as a $\Z_p$-module.
\end{fact}


\subsection{$\F_p$-cohomology of pro-$p$ groups}\label{ssec:H}
For an account on $\F_p$-cohomology of pro-$p$ groups, we refer to \cite[Ch.~I, \S~4]{serre:galc} and \cite[Ch.~III, \S~9]{nsw:cohn}.

Given a pro-$p$ group $G$, consider the abelian group $\F_p$ as a trivial $G$-module.
For $n\geq0$, we write simply $\rmH^n(G)$ for the $n$th cohomology group of $G$ with coefficients in the trivial $G$-module $\F_p$.
We recall that for a finitely generated pro-$p$ group $G$ one has the following:
\begin{itemize}
	\item[(a)] $\rmH^1(G)=\mathrm{Hom}(G,\F_p)$, and there is an isomorphism of finite $\F_p$-vector spaces
	\[
	\rmH^1(G)\simeq (G/\Phi(G))^\ast,
	\]
	where $\textvisiblespace^\ast$ denotes the $\F_p$-dual --- in particular, $\dim(\rmH^1(G))$ is the cardinality of a minimal generating set of $G$ --- (cf., e.g., \cite[Ch.~I, \S~4.2]{serre:galc});
	\item[(b)] $\dim(\rmH^2(G))$ is the cardinality of a minimal set of defining relations of $G$ (cf., e.g., \cite[Ch.~I, \S~4.3]{serre:galc}).
\end{itemize}

In particular, if $\mathcal{X}=\{x_1,\ldots,x_d\}$ is a minimal generating set of a finitely generated pro-$p$ group $G$, then $\rmH^1(G)$ has a basis $\{x_1^\ast,\ldots,x_d^\ast\}$, where $x_i^\ast(x_j)=\delta_{ij}$ for all $1\leq i,j\leq d$, and moreover $\alpha=\alpha(x_1)x_1^\ast+\ldots+\alpha(x_d)x_d^\ast$ for every $\alpha\in\rmH^1(G)$.

The $\F_p$-vector space
\[
\bfH^\bullet(G)=\coprod_{n\geq0}\rmH^n(G),
\]
endowed with the graded-commutative cup-product, is a graded $\F_p$-algebra (for the definition and the properties of the cup-product see, e.g., \cite[Ch.~I, \S~4--6]{nsw:cohn}).
Recall that $\bfH^\bullet(G)$ is said to be quadratic if the homomorphism of graded algebras
\begin{equation}
	\mu\colon \coprod_{n\geq0}\rmH^1(G)^{\otimes n}\longrightarrow \bfH^\bullet(G),
\end{equation}
which extends linearly $\alpha_1\otimes\cdots\otimes\alpha_m\mapsto \alpha_1\smallsmile\cdots\smallsmile\alpha_m$,
is surjective, and its kernel is generated as a graded ideal by a subspace of $\rmH^1(G)^{\otimes 2}$
(cf., e.g., \cite[\S~2.1]{qsv:quadratic}).

We will make use of the following proposition, which recaps the procedure to sort out the $\F_p$-cohomology ring of a pro-$p$ group with at most two defining relations, from the shape of these relations (cf. \cite[Prop.~4.2]{cq:onerel} and \cite[Exam.~2.7--2.8 and \S~3.2--3.3]{cq:2relUK}).

\begin{prop}\label{prop:quadratic few rel}
	Let $G$ be a finitely generated pro-$p$ group with 
$$\dim(\rmH^1(G))=d\geq2\qquad\text{and}\qquad1\leq \dim(\rmH^2(G))\leq2.$$
	Suppose that, for some ordered minimal generating set $\calX=\{x_1,\ldots, x_d\}$ of $G$, the defining relations of $G$ may be written as
	\[
	\begin{split}
		1=&\prod_{1\leq i<j\leq d} \left[x_i,x_j\right]^{a_{ij}}\cdot r_1'\qquad \text{with }0\leq a_{ij}\leq p-1\\
		1=&\prod_{1\leq i<j\leq d} \left[x_i,x_j\right]^{b_{ij}}\cdot r_2'\qquad \text{with }0\leq b_{ij}\leq p-1
	\end{split}
	\]
	--- the latter occurring if $\dim(\rmH^2(G))=2$ ---, with $a_{1,2}\neq0$, $b_{1,2}=0$, and $b_{hk}\neq0$ for at least a pair $(h,k)\neq(1,2)$; and
	\[r_1',r_2'\in\begin{cases} G^p\cdot[G',G], & \text{if }p\neq2,\\
		G^4\cdot (G')^2\cdot [G',G],&\text{if }p=2.
	\end{cases}\]
	Then one has the following:
\begin{itemize}
    \item[(a)]  $\rmH^1(G)$ has a basis $\{x_1^\ast\ldots,x_d^\ast\}$ such that 
	$$\mathcal{B}=\left\{\:x_1^\ast\smallsmile x_2^\ast,\text{ and }x_h^\ast\smallsmile x_k^\ast\text{ if }\dim(\rmH^2(G))=2\text{ and }b_{hk}\neq0\right\}$$ is a basis of $\rmH^2(G)$, and one has
	\[    x_i^\ast\smallsmile x_j^\ast=\frac{a_{ij}}{a_{1,2}}\left(x_1^\ast\smallsmile x_2^\ast\right)+\frac{b_{ij}}{b_{hk}}\left(x_h^\ast\smallsmile x_k^\ast\right)\qquad\text{for all }1\leq i<j\leq d;
	\]
    \item[(b)] $\cd(G)=2$ --- i.e., $\rmH^n(G)=0$ for $n\geq3$ ---, and $\bfH^\bullet(G)$ is a universally Koszul (hence quadratic) $\F_p$-algebra.
\end{itemize}
\end{prop}

\begin{rem}\label{rem:few rel}\rm
	The pro-$p$ groups in families $\mathcal{F}_1,\mathcal{F}_2$ satisfy the hypothesis of Proposition~\ref{prop:quadratic few rel}, as we shall see in \S\ref{ssec:G2 cohom} and \S\ref{ssec:G1 cohom}. Also, infinite Demu\v skin groups satisfy  the hypothesis of Proposition~\ref{prop:quadratic few rel} (with a single defining relation).
\end{rem}

Recall from the introduction that a pro-$p$ group $G$ is said to be a Bloch-Kato pro-$p$ group if the $\F_p$-cohomology algebra $\bfH^\bullet(H)$ is a quadratic $\F_p$-algebra for every subgroup $H$ of $G$.
By the Norm Residue Theorem, proved by V.~Voevodsky and M.~Rost (with the contribution of Ch.~Weibel), if $\K$ is a field containing a root of unity of order $p$, then the maximal pro-$p$ Galois group $G_{\K}(p)$ is a Bloch-Kato pro-$p$ group (cf., e.g., \cite[\S~2]{cq:bk}).

The following are well-known to be Bloch-Kato pro-$p$ groups.

\begin{exam}\label{ex:BK}\rm
	\begin{itemize}
		\item[(a)] A free pro-$p$ group is a Bloch-Kato pro-$p$ group (cf., e.g., \cite[Ex.~2.8]{qsv:quadratic}). Indeed, a free pro-$p$ group has cohomological dimension 1, so that its cohomology is trivially a quadratic $\F_p$-algebra, and every subgroup of a free pro-$p$ group is again a free pro-$p$ group (cf., e.g., \cite[Ch.~I,\S~4.2, Cor.~3]{serre:galc}).
		\item[(b)] A Demu\v skin group is a Bloch-Kato pro-$p$ group (cf., e.g., \cite[Cor.~5.7]{qw:cyc}). Indeed, the $\F_p$-cohomology algebra is quadratic by Proposition~\ref{prop:quadratic few rel}, and a closed subgroup of a Demu\v skin group is again a Demu\v skin group, if open, or a free pro-$p$ group, if not open (cf. \cite[Ch.~I, \S~4.5, Exer.~(5)--(6)]{serre:galc}).
	\end{itemize}
\end{exam}


\subsection{Massey products in Galois cohomology}\label{ssec:Massey}
For the formal definition and properties of Massey products in the $\F_p$-cohomology of pro-$p$ groups, we direct the reader to \cite{vogel,mt:Massey,cq:massey} and references therein.
Here we recall the few things needed to prove Proposition~\ref{prop:properties intro}--(2).

Given a pro-$p$ group $G$, consider a sequence $\alpha_1,\ldots,\alpha_n$ of length $n\geq2$ of (non-necessarily distinct) elements of $\rmH^1(G)$.
The associated $n$-fold Massey product, which is a subset of $\rmH^2(G)$, is denoted by $\langle\alpha_1,\ldots,\alpha_n\rangle$.
The $n$-fold Massey product $\langle\alpha_1,\ldots,\alpha_n\rangle$ is said
to be {\sl defined}, if $\langle\alpha_1,\ldots,\alpha_n\rangle\neq\varnothing$, and to {\sl vanish} if $0\in\langle\alpha_1,\ldots,\alpha_n\rangle$.
Finally, $G$ is said to satisfy the {\sl $n$-fold Massey vanishing property} if every $n$-fold Massey product vanishes whenever it is defined.

Massey products enjoy the following properties (cf., e.g.,  \cite[Prop.~1.2.3--1.2.4]{vogel} and \cite[Rem.~2.2]{mt:Massey}).

\begin{lem}\label{lem:masey}
	Let $G$ be a pro-$p$ group, and let $\alpha_1,\ldots,\alpha_n$ a sequence of elements of $\rmH^1(G)$ of length $n$, for $n\geq3$.
	\begin{itemize}
		\item[(a)] If $\alpha_i= 0$ for some $i$, then the $n$-fold Massey product $\langle\alpha_1,\ldots,\alpha_n\rangle$ vanishes.
		\item[(b)] 
		If the $n$-fold Massey product $\langle\alpha_1,\ldots,\alpha_n\rangle$ is defined, then $ \alpha_1\smallsmile\alpha',\alpha_n\smallsmile\alpha''\in\langle\alpha_1,\ldots,\alpha_n\rangle$
		for every $\alpha',\alpha''\in\rmH^1(G)$.
		\item[(c)] If the $n$-fold Massey product $\langle\alpha_1,\ldots,\alpha_n\rangle$ is defined, then
		\begin{equation}\label{eq:cup 0 massey}
			\alpha_1\smallsmile\alpha_2=\alpha_2\smallsmile\alpha_3=\ldots=\alpha_{n-1}\smallsmile\alpha_n=0
		\end{equation}
	\end{itemize}
\end{lem}

If the $n$-fold Massey product $\langle\alpha_1,\ldots,\alpha_n\rangle$ vanishes whenever the associated sequence of elements of $\rmH^1(G)$ of length $n$ satisfies \eqref{eq:cup 0 massey}, then $G$ satisfies the {\sl strong} $n$-fold Massey vanishing property (cf., e.g., \cite[Def.~1.2]{pal:Massey} or \cite[Def.~2.4]{cq:massey}).

\begin{exam}\label{ex:demushkin massey}\rm
	\begin{itemize}
		\item[(a)] If $G$ is a Demu\v skin group, $G$ satisfies the strong $n$-fold Massey vanishing property for every $n\geq3$ (cf., e.g., \cite[Thm.~3.5]{pal:Massey}).
		\item[(b)] If $\K$ is a field containing a root of unity of order $p$, then $G_{\K}(p)$ satisfies the $3$-fold Massey vanishing property (cf. \cite{eli:Massey,EM:Massey,MT:massey3}), and it is conjectured to satisfy the $n$-fold Massey vanishing property for every $n\geq3$ (cf. \cite{mt:conj}).
		It has been proved that this conjecture holds in several relevant cases (cf., e.g., \cite{mt:Massey, HW, palquick:Massey, cq:massey}).
	\end{itemize}
\end{exam}

Now let $n$ be a positive integer.
For $1\leq i,j\leq n$ let $E_{ij}\in\mathrm{Mat}_n(\F_p)$ denote the matrix with every entry equal to 0 but the $(i,j)$-entry, equal to 1, and let $I_n$ denote the identity matrix.
Let
$$\dbU_{n+1}=\left\{\: I_n+\sum_{1\leq i<j\leq n}a_{ij}E_{ij} \mid a_{i,j}\in\F_p\right\}\subseteq\mathrm{GL}_{n+1}(\F_p)$$
be the (multiplicative) $p$-group of upper unitriangular matrices with entries in $\F_p$.
The center of $\dbU_{n+1}$ is $\Zen(\dbU_{n+1})=\{I_n+aE_{1,n}\:\mid\:a\in\F_p\}$, which is isomorphic to $\F_p$.
We put $\bar\dbU_{n+1}=\dbU_{n+1}/\Zen(\dbU_{n+1})$.

Now consider a pro-$p$ group $G$, and let $\rho\colon G\to\dbU_{n+1}$ be a homomorphism of pro-$p$ groups.
For every $i=1,\ldots,n$ the $(i,i+1)$-entry of $\rho$, denoted by $\rho_{i,i+1}$, is a homomorphism $G\to\F_p$, and thus it may be considered as an element of $\rmH^1(G)$.
Analogously, if $n\geq2$ and $\bar\rho\colon G\to\bar\dbU_{n+1}$ is a homomorphism of pro-$p$ groups, then for every $i=1,\ldots,n$ the $(i,i+1)$-entry of $\bar\rho$, denoted by $\bar\rho_{i,i+1}$, is a homomorphism $G\to\F_p$, and thus it may be considered as an element of $\rmH^1(G)$ as well.
The following is the pro-$p$ translation of W.~Dwyer's interpretation (cf. \cite{dwyer}) of Massey products in terms of upper unitriangular representations (see, e.g., \cite[\S~8]{efrat:massey} or \cite[Lemma~9.3]{eq:kummer}).

\begin{prop}\label{prop:massey unip}\rm
	Let $G$ be a pro-$p$ group, and let $\alpha_1,\ldots,\alpha_n$ a sequence of elements of $\rmH^1(G)$ of length $n$, for $n\geq2$.
	\begin{itemize}
		\item[(a)] If there exists a homomorphism of pro-$p$ groups $\bar\rho\colon G\to\bar\dbU_{n+1}$ satisfying $\bar\rho_{i,i+1}=\alpha_i$ for every $i=1,\ldots,n$, then $\langle\alpha_1,\ldots,\alpha_n\rangle$ is defined.
		\item[(b)] If there exists a homomorphism of pro-$p$ groups $\rho\colon G\to\dbU_{n+1}$ satisfying $\rho_{i,i+1}=\alpha_i$ for every $i=1,\ldots,n$, then $\langle\alpha_1,\ldots,\alpha_n\rangle$ vanishes.
	\end{itemize}
\end{prop}

We will make use of the following well-known property of upper unitriangular matrices.

\begin{fact}\label{fact:matrix}
	Let $A=(a_{ij}),B=(b_{ij})\in\dbU_n$ such that $a_{ij}=0$ for if $0<j-i< m_1$, and $b_{ij}=0$ if  $0<j-i< m_2$, for some $m_1,m_2\geq1$, and set $C=(c_{ij}):=[A,B]$.
	Then $c_{ij}=0$ if $j-i<m_1+m_2$.
\end{fact}

\subsection{Graded group algebras and associated Lie algebras}\label{sec:gradedAlgebras}
For an introduction on restricted Lie algebras, we refer the reader to the original paper of Jacobson \cite{jacobson} on this topic. If $\mathfrak g$ is a restricted Lie $\F_p$-algebra with $p$-operation $x\mapsto x^{[p]}$, then its restricted universal envelope, denoted $u(\mathfrak g)$, is the quotient of the associative algebra $U(\mathfrak g):=T^\bullet(\mathfrak g)/(X\otimes Y-Y\otimes X-[X,Y]:X,Y\in \mathfrak g)$ by the ideal generated by the elements $X^{[p]}-X^{\otimes p}$, for $X$ belonging to the image of $\mathfrak g$ into $U(\mathfrak g)$. 

Given a pro-$p$ group $G$, denote by $\FppG$ the complete group algebra $\invlim_{U\lhd_o G}\F_p[G/U]$, and let $I_G$ be its augmentation ideal. The graded group algebra is the graded algebra associated with the $I_G$-adic filtration of $\FppG$, that is, \[\grFp G:=\coprod_{n\geq 0}I_G^n/I_G^{n+1}.\]
By a well known result of S.A.~Jennings, the subgroups $G_{n}:=\set{g\in G}{g-1\in I_G^n}$ coincide with the subgroups of the $p$-Zassenhaus filtration of $G$. It follows that the graded group algebra $\grFp G$ is the restricted universal envelope of the restricted Lie algebra \[\gr G:=\coprod_{n\geq 1}G_n/G_{n+1}\]
(cf., e.g., \cite[Ch.~12]{ddsms}, see also \cite[\S~3]{expo}).

Usually, describing the restricted Lie algebra $\gr G$ associated to a pro-$p$ group $G$, and its restricted universal envelope $\grFp G$, is a challenging task.
By Labute's work \cite{labute:mild}, if $G$ is mild, then it admits a presentation --- where the relations satisfy a condition of being independent in a maximal possible way --- from which a presentation
of the associated restricted Lie algebra may be deduced directly, cf. \cite[Thm.~1.2--(a)]{labute:mild}.
For instance, infinite Demu\v skin pro-$p$ groups are mild. Moreover, Proposition~\ref{prop:quadratic few rel} and \cite[Prop.~1.4]{expo} imply the following.

\begin{prop}\label{prop:mild}
    Let $G$ be a pro-$p$ group in one of the families $\mathcal{F}_1,\mathcal{F}_2$. Then $G$ is mild.
\end{prop}

A positively graded restricted Lie $\F_p$-algebra $\mathfrak g$ is said to be {\sl quadratic} (resp. {\sl Koszul}) if so is its restricted envelope $u(\mathfrak g)$ with respect to the induced grading, that is, \[\ext^{i,j}_{u(\mathfrak g)}(\F_p,\F_p)=0\quad \text{for $i,j\in \{1,2\}$, $i\neq j$ (resp. $\forall i\neq j$)}.\]
We say that $\mathfrak g$ is \textit{Bloch-Kato} (resp. \textit{weakly Bloch-Kato}) if all of its restricted subalgebras generated in degree one are Koszul (resp. quadratic).
If $p$ is an odd prime, then \cite[Cor. 2.3]{simone:koszul} shows that weakly Bloch-Kato restricted Lie algebras of characteristic $p$ are automatically Bloch-Kato.

Moreover, one has the following (cf. \cite[Thm. A]{simone:kurosh})
\begin{thm}
	Let $\mathfrak g$ be a positively graded restricted Lie algebra generated in degree one. Then $\mathfrak g$ is Bloch-Kato if and only if the $\F_p$-cohomology algebra $\bfH^\bullet(\mathfrak g,\F_p)$ is universally Koszul.
\end{thm}

In particular, the Lie theoretic version of Mina\v c's universal Koszulity conjecture holds true.

As an example, if $p$ is an odd prime and $G$ is a Demu\v skin group as in (\ref{eq:pres demushkin}), then the associated restricted Lie algebra admits presentation (see \cite{labute:onerel}) \[\gr G=\pres{x_1,y_1,\dots,x_d,y_d}{\sum_{i=1}^d[x_i,y_i]}\]
and is Bloch-Kato.

Indeed, all the proper subalgebras generated in degree one are free restricted Lie algebras (see \cite[\S 1.4]{simone:kurosh} for the ordinary case and \cite[\S 1.3, \S 4.1]{simone:koszul} for restricted Lie algebras). We call $\gr G$ the Demu\v skin Lie algebra of rank $2d$; if $d=1$, then the Lie algebra is a free-abelian restricted Lie algebra.

These Lie algebras play a special role in the theory of Bloch-Kato Lie algebras, because any Bloch-Kato restricted Lie algebra that is neither free nor abelian contains a Demu\v skin Lie algebra of positive rank as a subalgebra generated in degree one (\cite[\S 4]{simone:koszul}).
The group theoretic counterpart of this phenomenon is only conjectural, and part of I.~Efrat's Elementary Type Conjecture on maximal pro-$p$ Galois groups (see, e.g., \cite{etc_new}). 

Quadratic restricted Lie algebras with a single relation are free products of a free restricted Lie algebra and a Demu\v skin Lie algebra, and hence they are all Bloch-Kato. Actually, the class of quadratic Lie algebras which are automatically Bloch-Kato is much bigger, as the following result shows. 
Most Lie algebras considered in the present paper are instances of this class.

\begin{thm}\label{thm:2rel}
	Let $\mathfrak g$ be a quadratic restricted Lie algebra with at most two minimal relations, that is $\dim_{\,\F_p}\ext^2_{u(\mathfrak g)}(\F_p,\F_p)\leq 2$. 
	Then, $\mathfrak g$ is Bloch-Kato.
\end{thm}

A proof of this result can be found in \cite[Thm. 4.3]{simone:koszul}, where the first named author applies techniques developed by the second named author in \cite{cq:2relUK}.

The cohomology of a pro-$p$ group can be approximated via a spectral sequence discovered by J.P. May \cite{may:SS} involving the restricted cohomology of the associated graded restricted Lie algebra, that is, 
\[E_1^{i,j}=\ext^{i+j,j}_{u(\gr G)}(\F_p,\F_p)\implies \rmH^{i+j}(G,\F_p) \]

\begin{prop}
	Let $G$ be a pro-$p$ group whose associated restricted Lie algebra is Koszul. Then, the natural map $\hom(G,\F_p)\to \hom_{\F_p}(G/\Phi(G),\F_p)\simeq \hom_{\F_p}(\gr_1 G,\F_p)$ induces an isomorphism $\bfH^\bullet(G,\F_p)\to \ext^\bullet_{u(\gr G)}(\F_p,\F_p)$. In particular, $G$ is an $\rmH$-Koszul pro-$p$ group.
\end{prop}
\begin{proof}
    The differential of the $r$th page of May's spectral sequence is \[d_r^{i,j}:\ext_{u(\gr G)}^{i+j,j}(\F_p,\F_p)\to \ext_{u(\gr G)}^{i+j+1,j-r+1}(\F_p,\F_p).\] Since $\gr G$ is Koszul, for all $r\geq 1$, either the source or the target of $d_r$ are zero, and hence $d_r=0$, that is, the spectral sequence collapses at the first page $E_1=E_\infty$. Finally, Theorem 7.1 of \cite{pp:quadrAlg} implies that $\bfH^\bullet(G)$ is Koszul, and hence isomorphic to $\ext_{u(\mathfrak g)}^\bullet(\F_p,\F_p)$.
\end{proof}




\section{1-cyclotomic oriented pro-$p$ groups}\label{sec:1cyc}

\subsection{Orientations of pro-$p$ groups}\label{ssec:or}

Let
$$1+p\Z_p=\{\:1+p\lambda\:\mid\:\lambda\in\Z_p\:\}$$
denote the multiplicative group of principal units of the ring of integers $\Z_p$.
Given a pro-$p$ group $G$, a homomorphism of pro-$p$ groups $\theta\colon G\to1+p\Z_p$ is called an orientation of $G$.
Analogously, a pair $(G,\theta)$, with $\theta\colon G\to1+p\Z_p$ an orientation, is called an oriented pro-$p$ group.

An oriented pro-$p$ group $(G,\theta)$ gives rise to the continuous $G$-module $\Z_p(\theta)$, which is isomorphic to $\Z_p$ as an abelian pro-$p$ group, and with $G$-action
\[
x.\lambda=\theta(x)\cdot\lambda\qquad\text{for every }x\in G\text{ and }\lambda\in\Z_p(\theta).
\]
Since the action of $G$ on the quotient $G$-module $\Z_p(\theta)/p\Z_p(\theta)$ is trivial, we may identify $\Z_p(\theta)/p\Z_p(\theta)=\F_p$ as trivial $G$-modules.

Conversely, if a pro-$p$ group $G$ comes endowed with a continuous $G$ module $M$, isomorphic to $\Z_p$ as an abelian pro-$p$ group, then the action gives rise to an orientation $\theta\colon G\to1+p\Z_p$ such that $x.v=\theta(x)v$ for every $x\in G$ and $v\in M$.
Therefore, one may reformulate the definition of 1-cyclotomicity as follows.

\begin{defin}\label{defin:1cyc or}\rm
	An oriented pro-$p$ group $(G,\theta)$ is said to be Kummerian if the natural map
	\[
	\rmH^1(G,\Z_p(\theta)/p^n\Z_p(\theta))\longrightarrow\rmH^1(G),
	\]
	induced in cohomology by the epimorphism of continuous $G$-modules
	$$\xymatrix{\Z_p(\theta)/p^n\Z_p(\theta)\ar@{->>}[r] & \Z_p(\theta)/p\Z_p(\theta)=\F_p},$$
	is surjective.
	Moreover, $(G,\theta)$ is said to be 1-cyclotomic if $(H,\theta\vert_H)$ is a Kummerian oriented pro-$p$ group for every subgroup $H$ of $G$.
\end{defin}


\subsection{Maximal pro-$p$ Galois group}\label{ssec:GK 1cyc}
Oriented pro-$p$ groups arise naturally in Galois theory.

Let $\K$ be a field containing a root of unity of order $p$, and let $\mu_{p^\infty}\subseteq\K(p)^\times$ denote the group of the roots of 1 of order a power of $p$ lying in the maximal pro-$p$ Galois extension $\K(p)$ of $\K$.
The action of the maximal pro-$p$ Galois group $G_{\K}(p)$ on $\mu_{p^\infty}$ induces the pro-$p$ cyclotomic character
\[
\theta_{\K}\colon G_{\K}(p)\longrightarrow1+p\Z_p,
\]
satisfying $g(\zeta)=\zeta^{\theta_{\K}(g)}$ for all $g\in G_{\K}(p)$ and $\zeta\in\mu_{p^\infty}$.
Then $\Img(\theta_{\K})=1+p^f\Z_p$, where $f\in\dbN\cup\{\infty\}$ is maximal such that
$\K$ contains a root of unity of order $p^f$ (if $f =\infty$, we set $p^\infty = 0$).
Then $\theta_{\K}$ is a torsion-free orientation if $p\neq2$ or if $p=2$ and $\sqrt{-1}\in\K$.

The $G_{\K}(p)$-module $\Z_p(\theta_{\K})$ coincides with the 1st Tate twist $\Z_p(1)$ (cf. \cite[Def.~7.3.6]{nsw:cohn}), and moreover for every $n\geq1$ one has the equality of continuous $G_{\K}(p)$-modules
$$\Z_p(\theta_{\K})/p^n=\left\{\:\zeta\in\mu_{p^\infty}\:\mid\:\zeta^{p^n}=1\:\right\}.$$
By Kummer theory, the oriented pro-$p$ group $(G_{\K}(p),\theta_{\K})$ is Kummerian --- as shown {\sl ante litteram} by J.P.~Labute, see \cite[p.~131]{labute:demushkin}.

Observe that if $H$ is a subgroup of $G_{\K}(p)$, then $H=G_{\mathbb{L}}(p)$ for some algebraic pro-$p$ extension $\mathbb{L}/K$ (and thus clearly $\mathbb{L}$ contains the roots of 1 of order $p$), and moreover the pro-$p$ cyclotomic character $\theta_{\mathbb{L}}\colon H\to1+p\Z_p$ is the restriction of $\theta_{\K}$ to $H$, so that $(H,\theta_{\K}\vert_H)$ is Kummerian, and thus $(G_{\K}(p),\theta_{\K})$ is 1-cyclotomic.


\subsection{Detecting 1-cyclotomic oriented pro-$p$ groups}\label{ssec:detect}
An oriented pro-$p$ group $(G,\theta)$ has a distinguished normal subgroup
\[
K_\theta(G)=\left\langle\:y^x\cdot y^{-\theta(x)^{-1}}\:\mid\:x\in G,y\in\Ker(\theta)\:\right\rangle
\]
(cf., e.g., \cite[\S~3]{eq:kummer} or \cite[p.~2]{qw:bogomolov}).
We remark the following:
\begin{itemize}
	\item[(a)] $K_\theta(G)\subseteq\Ker(\theta)$;
	\item[(b)] moreover, $K_\theta(G)\supseteq \Ker(\theta)'$, so that $\Ker(\theta)/K_\theta(G)$ is abelian;
	\item[(c)] finally, if $\theta=\mathbf{1}_G$ --- i.e., $\theta$ is constantly equal to 1 --- then $K_\theta(G)$ is the derived subgroup $G'$.
\end{itemize}

One has the following group-theoretic characterization of Kummerian oriented pro-$p$ group (cf. \cite[Thm.~5.6]{eq:kummer}, see also \cite[Prop.~2.6]{qw:bogomolov}).

\begin{prop}\label{prop:K}
	Let $(G,\theta)$ be a finitely generated oriented pro-$p$ group, and suppose that $\Img(\theta)\subseteq1+4\Z_2$ if $p=2$.
	The following are equivalent:
	\begin{itemize}
		\item[(a)] $(G,\theta)$ is Kummerian;
		\item[(b)] $G/K_\theta(G)$ is a torsion-free pro-$p$ group;
		\item[(c)] $\Ker(\theta)/K_\theta(G)$ is a free abelian pro-$p$ group;
		\item[(d)] there is a minimal presentation $\pi\colon F\twoheadrightarrow G$ --- i.e., $F$ is a free pro-$p$ group and $\Ker(\pi)\subseteq\Phi(F)$ --- such that
		$$\Ker(\pi)\subseteq K_{\theta\circ\pi}(F).$$
	\end{itemize}
\end{prop}

In particular, if $\theta$ is constantly equal to 1, then $(G,\theta)$ is 1-cyclotomic if and only if $H^{\ab}$ is a free abelian pro-$p$ group for any closed subgroup $H$ of $G$.
Pro-$p$ groups with this property are called {\sl absolutely torsion-free pro-$p$ groups}, and were introduced by T.~W\"urfel in \cite{wurfel}.
Therefore, it is straightforward to deduce the following (cf., e.g., \cite[Rem.~2.3]{cq:galfeat})

\begin{exam}\label{ex:free abel 1cyc}\rm
	An oriented pro-$p$ group $(G,\theta)$, with $G$ a free abelian pro-$p$ group, is 1-cyclotomic if and only if $\theta$ is constantly equal to 1.
\end{exam}


\subsection{Demu\v skin pro-$p$ groups and their canonical orientation}\label{ssec:demushkin orientation}

J.P.~Labute showed in his PhD thesis that every Demu\v skin group $G$ comes endowed with a canonical orientation $\theta_G\colon G\to1+p\Z_p$ such that the associated oriented pro-$p$ group $(G,\theta_G)$ is Kummerian, and this is the only one completing $G$ into a Kummerian oriented pro-$p$ group  (cf. \cite[Thm.~4]{labute:demushkin}).
In particular, if $G$ has a presentation \eqref{eq:pres demushkin}, then
\[
\theta_G(y_1)=(1-q)^{-1}\qquad\text{and}\qquad\theta_G(x_1)=\ldots=\theta_G(x_d)=\theta_G(y_2)=\ldots=\theta_G(y_d)=1.
\]
Hence, a Demu\v skin group is determined --- up to isomorphism --- by the image of its canonical orientation together with the number of generators (cf. \cite[Ch.~I, \S~4.5, Rem.~2.a]{serre:galc}).

Moreover, one may show that the oriented pro-$p$ group $(G,\theta_G)$ is also 1-cyclotomic (cf. \cite[Thm.~5.5]{qw:cyc}).
Altogether, one has the following.

\begin{prop}\label{prop:orient demushkin}
	Let $G$ be a Demu\v skin group.
	The canonical orientation $\theta_G\colon G\to1+p\Z_p$ is the only orientation completing $G$ into a 1-cyclotomic oriented pro-$p$ group.
\end{prop}


\section{The family $\mathcal{F}_1$}\label{sec:G1}

In this section we focus on the first family of pro-$p$ groups $\mathcal{F}_1$, and $G$ will denote a pro-$p$ group with presentation
\[  G=\left\langle\:x_1,y_1,\ldots,x_d,y_d\:\mid\: \left[x_1^q,y_1\right]\cdots[x_{d},y_d]=1 \:\right\rangle\]
for some $d\geq2$ and some $p$-power $q=p^k$ with $k\in\{1,2,\ldots\}$ (and $k\geq2$ if $p=2$).

\subsection{Cohomology of $G$}\label{ssec:G2 cohom}
As underlined in Remark~\ref{rem:few rel}, $G$ satisfies the hypothesis of Proposition~\ref{prop:quadratic few rel}, considering $\{x_2,y_2,x_1,y_1,\ldots,y_d\}$ as an ordered minimal generating set.
In particular, $\{x_1^\ast,\ldots,y_d^\ast\}$ is a basis of $\rmH^1(G)$ and $\{x_2^\ast\smallsmile y_2^\ast\}$ a basis of $\rmH^2(G)$, and one has 
\[x_i^\ast\smallsmile y_i^\ast=x_2^\ast\smallsmile y_2^\ast\qquad\text{for all }i=3,\ldots,d,\]
while 
\[\begin{split}
	&   x_i^\ast\smallsmile x_j^\ast=y_i^\ast\smallsmile y_j^\ast=0\qquad\text{for all }1\leq i<j\leq d,\\
	& x_1^\ast\smallsmile y_1^\ast=x_i^\ast\smallsmile y_j^\ast=0\qquad \text{for all }i\neq j.
\end{split}\]
Also, the Euler characteristic of $G$ is  \[ E(G)=\sum_{n\geq0}(-1)^n\dim\left( \rmH^n(G)\right)=1-2d+1=2(1-d).\]

\begin{rem}\label{rem:BK G1}\rm
 Since $\cd(G)=2$, also $\cd(U)=2$ for every open subgroup $U$ of $G$ (cf. \cite[Ch.~I, \S~3.3, Prop.~14]{serre:galc}).
 Therefore, in order to check whether $\bfH^\bullet(U)$ is quadratic (and thus, whether $G$ is Bloch-Kato), it suffices to study $\rmH^2(U)$, and hence the defining relations of $U$.
\end{rem}



\subsection{The only candidate orientation}

Let $H$ be the subgroup of $G$ generated by $z=x_1^q$, and by $y_1,x_2,\ldots,y_d$.
Then $H$ is the Demu\v skin group
\[
H=\langle\:z,y_1,x_2,\ldots,y_d\:\mid\:[z,y_1][x_2,y_2]\cdots[x_d,y_d]=1\:\rangle.
\]
Suppose that $\theta\colon G\to1+p\Z_p$ is an orientation such that $(G,\theta)$ is a 1-cyclotomic oriented pro-$p$ group.
Then, also $(H,\theta\vert_H)$ is a 1-cyclotomic oriented pro-$p$ group, and hence necessarily
$$\theta(z)=\theta(y_1)=\ldots=\theta(y_d)=1$$
(cf. \S~\ref{ssec:demushkin orientation}).
Since $z=x_1^q$, one deduces that $\theta(x_1)^q=1$, and since $1+p\Z_p$ if $p\neq2$, and $1+4\Z_2$ if $p=2$, are torsion-free abelian pro-$p$ groups, one deduces that $\theta(x_1)=1$.

Altogether, the orientation $\theta$ must be constantly equal to 1, and therefore $G$ must be absolutely torsion-free (cf. \S~\ref{ssec:detect}).

\subsection{A subgroup of index $p$ and its abelianization}

Let $\phi\colon G\to\F_p$ be the homomorphism defined by $\phi(y_1)=1$ and
$$\phi(x_1)=\phi(x_2)=\ldots=\phi(x_d)=\phi(y_d)=0.$$
Then $U:=\ker\phi$ is a normal subgroup of $G$ of index $p$.
Now put $u=y_1^p$, and $z_h=x_1^{y_1^h}$, $x_i(h)=x_i^{y_1^h}$, and $y_i(h)=y_i^{y_1^h}$ for $h=0,\ldots,p-1$; $i=2,\ldots,d$.
Since $G/U=\{y_1^hU\:\mid\:h=0,\ldots,p-1\}$, the set
\[
\calX_U=\left\{\:u,z_h,x_i(h),y_i(h)\:\mid\:h=0,\ldots,p-1;\:i=2,\ldots,d\:\right\}
\]
generates $U$ as a pro-$p$ group.
More precisely, $U$ is the pro-$p$ group generated by $\calX_U$ and subject to the $p$ relations obtained conjugating the defining relation \eqref{eq:presG2} by $y_1^h$ for all $h=0,\ldots,p-1$.
Applying the equalities $[x_1^q,y_1]=z_0^{-q}\cdot z_1^q$ and
$$[x_1^q,y_1]^{y_1^{p-1}}=z_{p-1}^{-q}\cdot y_1^{-p}\cdot x_1^q\cdot y_1^p=z_{p-1}^{-q}\cdot z_0^q\cdot [z_0^q,u],$$
one obtains the $p$ relations
\begin{equation}\label{eq:rel U2 ab}
	\begin{split}
		& z_0^{-q}z_1^q[x_2(0),y_2(0)]\cdots[x_{d}(0),y_d(0)] = 1\\
		& z_1^{-q}z_2^q[x_2(1),y_2(1)]\cdots[x_{d}(1),y_d(1)] = 1 \\
		&\vdots\\
		& z_{p-1}^{-q}z_0^q[z_0^q,u][x_2(p-1),y_2(p-1)]\cdots
		[x_{d}(p-1),y_d(p-1)]=1.
	\end{split}
\end{equation}

By Remark \ref{rem:BK G1}, one has also $\rmH^n(U)=0$ for $n\geq 3$.
Observe that
\[
E(U)=1-|\calX_U|+p=1-(1+p+2p(d-1))+p=p(2-2d),
\]
and indeed $E(U)=p\cdot E(G)=[G:U]\cdot E(G)$
(cf. \cite[Ch.~I, \S~4.1, Exer.~(b)]{serre:galc}).
Altogether, one has the following.

\begin{prop}\label{prop:G2 no abs torfree}
	The pro-$p$ group $G$ is not absolutely torsion-free.
\end{prop}

\begin{proof}
	Consider the open subgroup $U$ of $G$, and for any element $w\in\calX_U$, let $\bar w=wU'$ denote the corresponding coset in $U^{\ab}$.
	Then $U^{\ab}$ is the abelian pro-$p$ group generated by
	\[
	\bar\calX_U=\left\{\:\bar u,\bar z_h,\overline{x_i(h)},\overline{y_i(h)}\:
	\mid\:h=0,\ldots,p-1,\:i=2,\ldots,d\:\right\},
	\]
	and, by \eqref{eq:rel U2 ab}, subject to the $p$ relations
	\[
	\left(\bar z_0^{-1}\bar z_1\right)^q=
	\left(\bar z_1^{-1}\bar z_2\right)^q=\ldots=
	\left(\bar z_{p-1}^{-1}\bar z_0\right)^q=1,
	\]
	as $[z_0^q,u],[x_i(h),y_i(h)]\in U'$ for every $h=0,\ldots,p-1$ and $i=2,\ldots,d$.
	Hence, $U^{\ab}$ is not torsion free --- in particular, it is isomorphic to $\Z_p^{2(d-1)p+1}\times (\Z/q)^p$.
\end{proof}

\begin{rem}\label{rem:U BK 1}\rm
	Since $\rmH^n(U)=0$ for $n\geq 3$, one may show --- employing \cite[Prop.~2.4--2.5 and Rem.~2.6]{qsv:quadratic} --- that the cohomology algebra $\bfH^\bullet(U)$ is a quadratic algebra: therefore, the subgroup $U$ may not be used to prove that $G$ is not Bloch-Kato.
\end{rem}



\subsection{Massey products}

\begin{prop}\label{prop:massey G2}
	If $n\leq q$ then $G$ satisfies the strong $n$-fold Massey vanishing property.    
\end{prop}

\begin{proof}
	Let $\alpha_1,\ldots,\alpha_n$ be a sequence of length $n$ of elements of $\rmH^1(G)$ satisfying \eqref{eq:cup 0 massey}.
	Since $n\leq q$, for any matrix $A\in\dbU_{n+1}$ one has $A^q\in\Zen(\dbU_{n+1})$.
	Therefore, if $A_1,B_1\in\dbU_{n+1}$ such that their $(h,h+1)$-entries, for $h=1,\ldots,n$, are equal to $\alpha_h(x_1)$ and $\alpha_h(y_1)$ respectively, then $[A_1^q,B_1]=I_{n+1}$.
	We will find matrices $A_2,\ldots,B_d\in\dbU_{n+1}$ such that 
	\begin{equation}\label{eq:comm A2B2 bla}
		[A_2,B_2]\cdots[A_d,B_d]=I_{n+1},
	\end{equation}
	and for all $i=2,\ldots,d$ the $(h,h+1)$-entries of $A_i$ and $B_i$ are equal to $\alpha_h(x_i)$ and $\alpha_h(y_i)$ respectively, 
	so that the assignment $x_i\mapsto A_i$ and $y_i\mapsto B_i$, $i=1,\ldots,d$, yields a morphism $\rho\colon G\to\dbU_{n+1}$ satisfying the properties prescribed in Proposition~\ref{prop:massey unip}--(b).
	
	Let $N$ be the normal subgroup of $G$ generated by $x_1,y_1$ as a normal subgroup, set $\bar G=G/N$ --- namely, $\bar G$ is the Demu\v skin group 
	\[\bar G=\left\langle\,\bar x_2,\,\bar y_2,\ldots,\,\bar x_d,\,\bar y_d\:\mid\:[\bar x_2,\bar y_2]\cdots[\bar x_d,\bar y_d]=1\:\right\rangle,\]
	where $\bar w$ denotes the coset $wN\in\bar G$, for $w\in G$ ---
	and let $\pi\colon G\to\bar G$ denote the canonical orientation.
	Also, for $\beta\in\rmH^1(G)$, set 
	$$\bar \beta=\beta(x_2)\bar x_2^\ast+\beta(y_2)\bar y^\ast_2+\ldots+\beta(x_d)\bar x_d^\ast+\beta(y_d)\bar y^\ast_d\in\rmH^1(\bar G).$$
	Since $\alpha_1,\ldots,\alpha_n$ satisfy \eqref{eq:cup 0 massey}, one has 
	\[
	0=\alpha_h\smallsmile\alpha_{h+1}=\left(\sum_{i=2}^d\left(\alpha_h(x_i)\alpha_{h+1}(y_i)-\alpha_h(y_i)\alpha_{h+1}(x_i)\right)\right)(x_2^\ast\smallsmile y_2^\ast),
	\]
	and hence the sum above is 0. 
	Since $\bar G$ is a Demu\v skin group,
	\[\begin{split}
		\bar \alpha_h\smallsmile\bar \alpha_{h+1} &=
		\sum_{i=2}^d\left(\alpha_h(x_i)\alpha_{h+1}(y_i)-\alpha_h(y_i)\alpha_{h+1}(x_i)\right)(\bar x_i^\ast\smallsmile \bar y_i^\ast)
		\\
		&=\left(\sum_{i=2}^d\left(\alpha_h(x_i)\alpha_{h+1}(y_i)-\alpha_h(y_i)\alpha_{h+1}(x_i)\right)\right)(\bar x_2^\ast\smallsmile \bar y_2^\ast),
	\end{split}\]
	and hence also the sequence $\bar \alpha_1,\ldots,\bar \alpha_n$ of elements of $\rmH^1(\bar G) $ satisfies \eqref{eq:cup 0 massey}.
	Therefore, by Proposition~\ref{prop:massey unip}--(b) there exists a homomorphism $\eta\colon \bar G\to\dbU_{n+1}$ such that $$\eta_{h,h+1}= \bar \alpha_h\qquad\text{for all }h=1,\ldots,n-1,$$
	as the Demu\v skin group $\bar G$ has the strong $n$-fold Massey vanishing property, cf. Example~\ref{ex:demushkin massey}--(a).
	Now set $A_i=\eta(\bar x_i)$ and $B_i=\eta(\bar y_i)$ for $i=2,\ldots,d$.
	Then these matrices satisfy \eqref{eq:comm A2B2 bla}, and this completes the proof.
\end{proof}


\subsection{The graded group algebra}
Consider the graded group algebra $\gr \FppG$, generated by $X_1,Y_2\ldots,X_d,Y_d$.
The initial form of the defining relation of $G$ --- i.e., the image in $G_2/G_3$ --- is the Lie polynomial
\[
(X_2Y_2-Y_2X_2)+\ldots+(X_dY_d-Y_dX_d).
\]
Since $G$ is a mild pro-$p$ group by Proposition~\ref{prop:mild},
\cite[Thm.~2.12--(3)]{Jochen} implies that
\[\begin{split}
    \grFp G&\simeq \frac{T^\bullet(X_1,Y_1,\dots,X_d,Y_d)}{([X_2,Y_2]+\dots+[X_d,Y_d])}\\&=T^\bullet(X_1,Y_1)\sqcup \frac{T^\bullet(X_2,Y_2,\dots,X_d,Y_d)}{([X_2,Y_2]+\dots+[X_d,Y_d])},
\end{split}\]
where the latter is a free product of quadratic $\F_p$-algebras.
In particular, Proposition~\ref{prop:properties intro}--(3) follows by \cite[Thm.~1.3]{expo}.

By Jennings' theorem, the restricted Lie algebra $\gr G$ is the free product of a free Lie algebra of rank $2$ --- generated by $X_1$ and $Y_1$ --- and a \textit{Demu\v skin Lie algebra} on $2d-2$ generators, namely, the restricted Lie algebra associated to a Demu\v skin group on $2d-2$ generators (with $q\neq2$). By \cite[Thm.~4.3]{simone:kurosh} and \cite[\S~1.3]{simone:koszul}, the restricted Lie algebra $\gr G$ is Bloch-Kato, with cohomology $\ext^\bullet_{u(\gr G)}(\F_p,\F_p)\simeq \bfH^\bullet(G)$. 
This proves Proposition~\ref{prop:gradedalgebras intro} for $\mathcal F_1$.


\section{The family $\mathcal{F}_2$}

In this section we focus on the pro-$p$ groups in the family $\mathcal{F}_2$, and $G$ will denote a pro-$p$ group with presentation
\[  G=\left\langle\:x_1,y_1,\ldots,x_d,y_d\:\mid\: x_1^q[x_1,y_1]\cdots[x_{d-1},x_d]=[z_1,z_2]=1 \:\right\rangle\]
where $d\geq2$, $q=p^k$ with $k\in\{1,2,\ldots,\infty\}$ (and $k\geq2$ if $p=2$), and $z_1,z_2\in\calX_G$, $\{z_1,z_2\}\neq\{x_i,y_i\}$ for any $i=1,\ldots,d$.
--- here we set $\calX_G=\{x_1,y_1,\ldots,y_d\}$.

Observe that $G$ is the quotient of the Demu\v skin group
\[
\tilde G=\left\langle\:x_1,y_1,\ldots,x_d,y_d\:\mid\: x_1^q[x_1,y_1]\cdots[x_{d-1},x_d]=1 \:\right\rangle
\]
over the normal subgroup generated by $z_1,z_2$ --- we will set $\pi\colon \tilde G\to G$ the canonical projection.


\subsection{Cohomology of $G$}\label{ssec:G1 cohom}
As underlined in Remark~\ref{rem:few rel}, $G$ satisfies the hypothesis of Proposition~\ref{prop:quadratic few rel}, considering $\calX_G=\{x_1,y_1,\ldots,y_d\}$ as an ordered minimal generating set.
In particular, $\{x_1^\ast,\ldots,y_d^\ast\}$ is a basis of $\rmH^1(G)$ and $\{x_2^\ast\smallsmile y_2^\ast,z_1^\ast\smallsmile z_2^\ast\}$ a basis of $\rmH^2(G)$, and one has 
\[x_i^\ast\smallsmile y_i^\ast=x_1^\ast\smallsmile y_1^\ast\qquad\text{for all }i=2,\ldots,d,\]
while 
\[ x_i^\ast\smallsmile x_j^\ast=y_i^\ast\smallsmile y_j^\ast= x_i^\ast\smallsmile y_j^\ast=0\qquad \text{for all }i\neq j
\]
if the two factors in the cup-products above are not $z_1,z_2$.

Also, the Euler characteristic of $G$ is  \[ E(G)=\sum_{n\geq0}(-1)^n\dim\left( \rmH^n(G)\right)=1-|\calX_G|+2=3-2d.\]

\begin{rem}\label{rem:BK G2}\rm
As observed in Remark \ref{rem:BK G1} for $\mathcal{F}_1$, since $\cd(G)=2$, also $\cd(U)=2$ for every open subgroup $U$ of $G$.
 Therefore, in order to check whether $\bfH^\bullet(U)$ is quadratic (and thus, whether $G$ is Bloch-Kato), it suffices to study $\rmH^2(U)$, and hence the defining relations of $U$.
\end{rem}


\subsection{The only candidate orientation}\label{ssec:G1 only orient}

Since $G$ is a quotient of the Demu\v skin group $\tilde G$, if $\theta\colon G\to 1+p\Z_p$ is an orientation turning $G$ into a 1-cyclotomic oriented pro-$p$ group, then such an orientation must be ``inherited'' from $\tilde G$, as shown by the following.

\begin{thm}
	Let $(G,\theta)$ be a Kummerian oriented pro-$p$ group, and let $p:\tilde G\to G$ be a Frattini cover, i.e., $\ker (p)\subseteq \Phi(\tilde G)$. Then, $(\tilde G,\theta\circ p)$ is Kummerian.
\end{thm}
\begin{proof}
	Let $\tilde\pi:F\to \tilde G$ be a minimal presentation of $\tilde G$, and consider the composite $\pi=p\circ \tilde \pi:F\to  G$, which is a minimal presentation for $G$. 
	By Kummerianity of $(G,\theta)$, one has $\ker(\pi)\subseteq K_{\theta\circ\pi}(F)$, and hence, by construction, $$\ker(\tilde\pi)\subseteq \ker(\pi)\subseteq K_{\theta\circ \pi}(F)=K_{\theta\circ p\circ \tilde\pi}(F),$$ which implies that $(\tilde G,\theta\circ p)$ is Kummerian. 
\end{proof}

\begin{prop}\label{prop:G1 only orient}
	If $\theta\colon G\to 1+p\Z_p$ is an orientation such that $(G,\theta)$ is a 1-cyclotomic oriented pro-$p$ group with $G\in \mathcal F_2$, then $\theta\circ\pi\colon\tilde G\to1+p\Z_p$ is the canonical orientation $\theta_{\tilde G}$ of the Demu\v skin group $\tilde G$.
	Moreover, $z_1,z_2\neq y_1$.
\end{prop}

\begin{proof}
	Let $H$ be the subgroup of $G$ generated by $z_1,z_2$.
	Since $[z_1,z_2]=1$, $H$ is an abelian pro-$p$ group.
	On the other hand, $\cd(G)\leq 2<\infty$, so that $G$ --- and thus $H$ --- is torsion-free.
	Altogether, $H$ is a free abelian pro-$p$ group, and thus $\theta\vert_H$ is constantly equal to 1 by Example~\ref{ex:free abel 1cyc}.
	
	Now set $\theta_0=\theta\circ\pi$.
	Then $z_1,z_2$ --- considered as elements of $\tilde G$ --- lie in $\Ker(\theta_0)$, and hence $[z_1,z_2]\in K_{\theta_0}(\tilde G)$.
	Thus,
	\[
	\Ker(\theta_0)/K_{\theta_0}(\tilde G)\simeq \Ker(\theta)/K_\theta(G),
	\]
	and the latter is a free abelian pro-$p$ group by Proposition~\ref{prop:K}, as $(G,\theta)$ is Kummerian.
	Consequently, also $(\tilde G,\theta_0)$ is a Kummerian oriented pro-$p$ group, again by Proposition~\ref{prop:K}.
	
	Therefore, $\theta_0$ must be the canonical orientation $\theta_{\tilde G}\colon\tilde G\to1+p\Z_p$ of the Demu\v skin group $\tilde G$ (cf. \S~\ref{ssec:demushkin orientation}).
	Hence,
	\[
	\theta(y_1)=(1-q)^{-1}\qquad\text{and}\qquad\theta(x_1)=\ldots=\theta(x_d)=\theta(y_2)=\ldots=\theta(y_d)=1.
	\]
	In particular, $z_1,z_2\neq y_1$, as $z_1,z_2\in\Ker(\theta)$.
\end{proof}


\subsection{A subgroup of index $p$}
Let $1\leq i_1,i_2\leq d$ be the indices such that $z_1\in\{x_{i_1},y_{i_1}\}$ and $z_2\in\{x_{i_2},y_{i_2}\}$ --- without loss of generality we may assume that $i_1<i_2$ ---, and
set $t_1,t_2\in\calX_G$ such that $\{z_1,t_1\}=\{x_{i_1},y_{i_1}\}$, and $\{z_2,t_2\}=\{x_{i_2},y_{i_2}\}$. 
In other words, $z_1,t_1$ appear together in a commutator in the first relation of $G$, and analogously $z_2,t_2$.
After renaming the generators, the first relation of $G$ is
\begin{equation}\label{eq:G1 rel1 1caso}
	z_1^q[z_1,t_1][x_2,y_2]\cdots[z_2,t_2]^{\epsilon_2}\cdots[x_d,y_d]=1,
\end{equation}
if $z_1=x_1$; or
\begin{equation}\label{eq:G1 rel1 2caso}
	x_1^q[x_1,y_1]\cdots[z_1,t_1]^{\epsilon_1}\cdots[z_2,t_2]^{\epsilon_2}\cdots[x_d,y_d]=1,
\end{equation}
if $i_1>1$.
Here $\epsilon_1\in\{1,-1\}$ and its value depends on whether $z_1=x_{i_1}$ and $t_1=y_{i_1}$ or conversely, and analogously with $\epsilon_2$ and $z_2,t_2$.

Now let $\phi_U\colon G\to\F_p$ be the homomorphism of pro-$p$ groups defined by $\phi_U(z_2)=1$, and $\psi_U(w)=0$ for any other generator $w\in\calX_G$, $w\neq z_2$; and set $U=\Ker(\phi_U)$.
Then $[G:U]=p$, and hence
\[
E(U)=[G:U]\cdot E(G)=p(3-2d)
\]
(cf. \cite[Ch.~I, \S~4.1, Exer.~(b)]{serre:galc}).
Put $u=z_2^p$, and $w(h)=w^{z_2^h}$ for $w\in\calX_G$, $w\neq z_2$.
Since $G/U=\{x_1^hU\:\mid\:h=0,\ldots,p-1\}$, by the definition of $\phi_U$ the set
\[
\calX_U=\left\{\:u,\:w(h)\:\mid\:w\in\calX_G,\:w\neq z_2,\:0\leq h\leq p-1\:\right\}
\]
generates $U$ as a pro-$p$ group --- not minimally, as shown by the following.

\begin{lem}\label{YU A}
	The set \[
	\calY_U=\left\{\:u,z_1,t_2,t_1(h),x_i(h),y_i(h),\:\mid\:i\neq i_1,i_2,\:0\leq h\leq p-1\:\right\},
	\]
	is a minimal generating set of $U$ as a pro-$p$ group.
	Moreover, $\dim(\rmH^2(U))=2$.
\end{lem}

\begin{proof}
	By the first relation of $G$ one has
	\[
	\left(z_2^{-1}\cdot t_2^{z_2}\right)^{-\epsilon_2}=[z_2,t_2]^{-\epsilon_2}=\left([x_{i_2+1},y_{i_2+1}]\cdots[x_{d-1},x_d]\right)\cdot\left(z_1^q[z_1,t_1]\cdots\right),
	\]
	from \eqref{eq:G1 rel1 1caso} (i.e., if $z_1=x_1$); or
	\[
	\left(z_2^{-1}\cdot t_2^{z_2}\right)^{-\epsilon_2}=[z_2,t_2]^{-\epsilon_2}=\left([x_{i_2+1},y_{i_2+1}]\cdots[x_{d},y_d]\right)\cdot\left(x_1^q[x_1,y_1]\cdots[z_1,t_1]^{\epsilon_1}\cdots\right).
	\]
	from \eqref{eq:G1 rel1 2caso}.
	Thus, $t_2^{z_2}$ --- and hence also $t_2(h)$ for every $0\leq h\leq p-1$ --- is contained in the pro-$p$ group generated by $\calY_U$.
	Also, by the second relation of $G$ one has $z_1(h)=z_1$ for every $0\leq h\leq p-1$.
	Altogether, the whole set $\calX_U$ is contained in the pro-$p$ group generated by $\calY_U$, and hence $\calY_U$ generates $U$ as a pro-$p$ group.
	
	Therefore, $$\dim(\rmH^1(U))\leq |\calY_U|=3+p+p(2(d-2))=3-p(3-2d).$$
	On the other hand, one has
	\[
	\begin{split}
		p(3-2d)=E(U) &= 1-\dim\left(\rmH^1(U)\right)+\dim\left(\rmH^2(U)\right)\\
		&\geq 1-(3-p(3-2d))+\dim\left(\rmH^2(U)\right),
	\end{split}\]
	so that $\dim(\rmH^2(U))\leq 2$.
	Since the pro-$p$ group $U$ is subject to the two independent defining relations
	\begin{equation}\label{eq:rel1U A}
		[u,z_1]=1
	\end{equation}
	and
	\begin{equation}\label{eq:rel2U A}
		\begin{split}
			[t_2,u] &= \prod_{h=0}^{p-1}[t_2,z_2]^{z_2^h}\\
			&=\prod_{h=0}^{p-1}\left(\left([x_{i_2+1}(h),y_{i_2+1}(h)]\cdots[x_{d-1}(h),x_d(h)]\right)\cdot \left(z_1^q[z_1,t_1(h)]\cdots\right)\right)^{\epsilon_2},
		\end{split}
	\end{equation}
	if $z_1=x_1$, or
	\begin{equation}\label{eq:rel2U B}
		\begin{split}
			[t_2,u] &= \prod_{h=0}^{p-1}[t_2,z_2]^{z_2^h}\\
			&=\prod_{h=0}^{p-1}\left(\left([x_{i_2+1}(h),y_{i_2+1}(h)]\cdots[x_{d}(h),y_d(h)]\right)\cdot\left(\cdots[z_1,t_1(h)]^{\epsilon_1}\cdots\right)\right)^{\epsilon_2},
		\end{split}
	\end{equation}
	one deduces that $\dim(\rmH^2(U))\geq2$.
	
	Altogether, $\dim(\rmH^2(U))=2$, and consequently $\dim(\rmH^1(U))$ is equal to the cardinality of $\calY_U$, which yields the first claim.
\end{proof}

\begin{rem}\label{rem:G2 U kummer}\rm
	\begin{itemize}
		\item[(a)] By Remark~\ref{rem:few rel}, the cohomology algebra $\bfH^\bullet(U)$ is quadratic and universally Koszul.
		\item[(b)] By Proposition~\ref{prop:G1 only orient} and by \S~\ref{ssec:demushkin orientation}, assuming that $(G,\theta)$ is 1-cyclotomic implies that $\theta\vert_U(y_1(h))=(1-q)^{-1}$ for all $0\leq h\leq p-1$, and $\theta\vert_U(w)=1$ for all $w\in\calY_U$, $w\neq y_1(h)$.
		Hence, all factors appearing in \eqref{eq:rel1U A}--\eqref{eq:rel2U B} belong to $K_{\theta\vert_U}(U)$ --- in particular, one has
		\[\begin{split}
			x_1(h)^q[x_1(h),y_1(h)] &= x_1(h)^{q-1}\cdot x_1(h)^{y_1(h)}   \\
			&=x_1(h)^{-\theta(y_1(h))^{-1}}\cdot x_1(h)^{y_1(h)}\in K_{\theta\vert_U}(U),
		\end{split}\]
		if, $i_1>1$,
		and analogously $z_1^q[z_1,t_1(h)]\in K_{\theta\vert_U}(U)$, if $i_1=1$ (and hence $z_1=x_1(h)$ and $t_1(h)=y_1(h)$ for all $h$).
		Also, $z_1,u\in\Ker(\theta)$.
		Therefore, $\Ker(\theta\vert_U)/K_{\theta\vert_U}(U)$ is a free abelian pro-$p$ group, and $(U,\theta\vert_U)$ is Kummerian by Proposition~\ref{prop:K}.
	\end{itemize}
\end{rem}


\subsection{A subgroup of index $p^2$}

By Remark~\ref{rem:G2 U kummer}, the subgroup $U$ of $G$ does not prove that $(G,\theta)$ is not 1-cyclotomic (or Bloch-Kato).
Hence, we carry on and investigate a suitable subgroup of index $p$ of $U$.

So, let $\phi_V\colon U\to\F_p$ be the homomorphism of pro-$p$ groups defined by
\[
\phi_V(t_1(h))=1\qquad\text{for all }h=0,\ldots,p-1,
\]
and $\phi_V(w)=0$ for every $w\in\calY_U$, $w\neq t_1(h)$.
Then $[U:V]=p$, and hence
$$E(V)=p\cdot E(U)=p^2\cdot E(G)=p^2(3-2d).$$
Also, put $v=t_1^p$, and $w_h:=t_1(h)\cdot t_1^{-1}$ for every $h=1,\ldots, p-1$.
Then $t_1(h)=w_ht_1$, and
\begin{equation}\label{eq:comm_yt}
	[z_1,t_1(h)]=[z_1,w_ht_1]=[z_1,t_1]\cdot[z_1,w_h]^{t_1}\qquad\text{for all }h=1,\ldots,p-1.
\end{equation}

Since $U/V=\{t_1^kV\:\mid\:k=0,\ldots,p-1\}$, the set
\[
\calX_V=\left\{\:v,\:w_1^{t_1^k},\ldots,w_{p-1}^{t_1^k},\:s^{t_1^k},\:\mid\:0\leq k<p,\:s\in\calY_U,s\neq t_1(h)\:\right\}
\]
generates $V$ as a pro-$p$ group.
We claim that $\calX_V$ is a minimal generating set of $V$.

\begin{lem}\label{lem:G1 Vab gen}
	The set $\calX_V$ is a minimal generating set of $V$ as a pro-$p$ group.
\end{lem}

\begin{proof}
	For every $s\in V$ let $\bar s=sV'\in V^{\ab}$ denote its coset.
	By Fact~\ref{fact}, it suffices to show that the set of cosets
	$$\calX_{V^{\ab}}=\{\:sV'\:\mid\:s\in\calX_V\:\}$$ is a minimal generating set of $V^{\ab}$.
	
	Clearly, $\calX_{V^{\ab}}$ generates $V^{\ab}$ as an abelian pro-$p$ group.
	Since $u,z_1\in V=\Ker(\phi_V)$, the defining relation \eqref{eq:rel1U A} yields trivial relations in the abelian pro-$p$ group $V^{\ab}$ --- namely, $[\bar u(k),\bar z_1(k)]=1$ for all $0\leq k\leq p-1$.
	
	On the other hand, all the elements of $\calX_U$ appearing in the relations \eqref{eq:rel2U A}--\eqref{eq:rel2U B} lie in $V=\Ker(\phi_V)$, but the $t_1(h)$'s.
	Hence, all the elementary commutators appearing in these two relations, but those with the $t_1(h)$'s, yield trivial elements of $V^{\ab}$, and therefore, by equality~\ref{eq:comm_yt} --- and by the fact that $z_1,w_h\in V$ ---, relation~\eqref{eq:rel2U A} gives the $p$ relations
	\[
	1=\prod_{h=0}^{p-1}\overline{\left(\left( z_1^q[z_1,t_1]\right)^{\epsilon_2}\right)^{t_1^k}}=\left(\left(\overline{z_1^{t_1^k}}\right)^{q-1}\cdot \overline{z_1^{t_1^{k+1}}}\right)^{\epsilon_2p},\qquad 0\leq k\leq p-1;
	\]
	while relation~\eqref{eq:rel2U B} gives the $p$ relations
	\[
	1=\prod_{h=0}^{p-1}\overline{\left(\left( [z_1,t_1]\right)^{\epsilon_2}\right)^{t_1^k}}=\left(\left(\overline{z_1^{t_1^k}}\right)^{-1}\cdot \overline{z_1^{t_1^{k+1}}}\right)^{\epsilon_2 p},
	\qquad 0\leq k\leq p-1\]
	--- observe that $\overline{z_1^{t_1^p}}=\bar z_1$, as $t_1^p=v\in V$.
	Both the above relazions in $V^{\ab}$ involve $p$-powers of elements of $\calX_{V^{\ab}}$, and thus $\calX_{V^{\ab}}$ minimally generates $V^{\ab}$ as an abelian pro-$p$ group.
	In particular, \[V^{\ab}\simeq \Z_p^{|\calX_{V^{\ab}}|-p}\times (\F_p)^p.\qedhere\]
\end{proof}

Since in the proof of Lemma~\ref{lem:G1 Vab gen} it is shown that the abelian pro-$p$ group $V^{\ab}$ is not torsion-free, one obtains the following (cf. \S~\ref{ssec:detect}).

\begin{cor}\label{cor:G1 no abstorfree}
	The pro-$p$ group $G$ is not absolutely torsion-free, and hence the oriented pro-$p$-group $(G,\mathbf{1}_G)$, where $\mathbf{1}_G\colon G\to\{1\}\subseteq 1+p\Z_p$ denotes the trivial orientation, is not 1-cyclotomic.
	In particular, if $q=0$ then $G$ may not complete into a 1-cyclotomic oriented pro-$p$ group.
\end{cor}

\begin{rem}\label{rem:G2 V quad}\rm
	By Lemma~\ref{lem:G1 Vab gen}, one has
	$$\dim(\rmH^1(V))=|\calX_V|=1+p(p-1)+p(|\calY_U|-p)=1-p+p\dim(\rmH^1(U)).$$
	On the other hand,
	$$1-\dim(\rmH^1(V))+\dim(\rmH^2(V))=E(V)=pE(U)=3p-p\dim(\rmH^1(U)).$$
	It follows that $\dim(\rmH^2(V))=2p$.
	Therefore, the $2p$ defining relations of $V$ are the $p$ conjugates by $t_1^k$, $0\leq k\leq p-1$, of \eqref{eq:rel1U A}, and of \eqref{eq:rel2U A} or \eqref{eq:rel2U B}.
	Since $\rmH^n(G)=0$ for $n\geq 3$, one has also $\rmH^n(V)=0$ for $n\geq 3$ (cf. \cite[Ch.~I, \S~3.3, Prop.~14]{serre:galc}), and one may show --- employing \cite[Prop.~2.4--2.5 and Rem.~2.6]{qsv:quadratic} --- that the cohomology algebra $\bfH^\bullet(V)$ is a quadratic algebra: therefore, the subgroup $V$ may not be used to prove that $G$ is not Bloch-Kato.
\end{rem}


\subsection{The subgroup $V$ and Kummerianity}
Recall that by Proposition~\ref{prop:G1 only orient}, there is only a possible orientation $\theta\colon G\to1+p\Z_p$ turning $G$ into a 1-cyclotomic oriented pro-$p$ group $(G,\theta)$. In particular, $$\theta(t_1(h))=(1-q)^{-1}\qquad\text{for all }0\leq h\leq p-1,$$ and thus $\theta(v)=(1-q)^{-p}$, and $\theta(s)=1$ for all other $s\in\calX_V$.
Since the case $\theta=\mathbf{1}_G$ has already been ruled out by Corollary~\ref{cor:G1 no abstorfree}, we focus now on the opposite case.

\begin{prop}\label{prop:Vab A}
	If $\theta\neq\mathbf{1}_G$, the oriented pro-$p$ group abelian pro-$p$ group $(V,\theta\vert_V)$ is not Kummerian, and hence $(G,\theta)$ is not 1-cyclotomic.
\end{prop}

\begin{proof}
	Our goal is to show that the quotient $\bar V=V/K_{\theta\vert_{V}}(V)$ is not torsion-free, so that $(V,\theta_V)$ is not Kummerian by Proposition~\ref{prop:K}.
	
	Within this proof, for every $s\in V$ let $\bar s=sK_{\theta\vert_{V}}(V)\in\bar V$ denote its coset.
	Since $K_{\theta\vert_{V}}(V)\subseteq\Phi(V)$, the set $\calX_{\bar V}=\{\bar s\:\mid\:s\in\calX_V\}$ minimally generates $\bar V$ as a pro-$p$ group by Lemma~\ref{lem:G1 Vab gen}.
	
	Similarly to the proof of Lemma~\ref{lem:G1 Vab gen}, if the second defining relation of $U$ is \eqref{eq:rel2U A} --- namely, $z_1=x_1$ ---, $V$ satisfies the relations
	\[ \left[t_2^{t_1^k},u^{t_1^k}\right] = \prod_{h=0}^{p-1}\left(\left[x_{i_2+1}(h)^{t_1^k},y_{i_2+1}(h)^{t_1^k}\right]\cdots\right)
	\left(\left(z_1^{t_1^k}\right)^{q-1} z_1^{t_1^{k+1}}\left[z_1^{t_1^k},w_h^{t_1^k}\right]^{t_1}\cdots\right)^{\epsilon_2} \]
	for every $0\leq k\leq p-1$; while if the second defining relation of $U$ is \eqref{eq:rel2U B} --- namely, $i_1>1$ ---, $V$ satisfies the relations
	\[ \left[t_2^{t_1^k},u^{t_1^k}\right] = \prod_{h=0}^{p-1}\left(\left[x_{i_2+1}(h)^{t_1^k},y_{i_2+1}(h)^{t_1^k}\right]\cdots\right)
	\left(\cdots\left(z_1^{t_1^k}\right)^{-1} z_1^{t_1^{k+1}}\left[z_1^{t_1^k},w_h^{t_1^k}\right]^{t_1}\cdots\right)^{\epsilon_2} \]
	for every $0\leq k\leq p-1$ --- in both cases we have applied equality \eqref{eq:comm_yt}.
	Since $t_2,u,z_1,w_h,x_j(h),y_j(h)\in\Ker(\theta\vert_V)$ for all $j\neq i_1,i_2$ and $0\leq h\leq p-1$, all the comutators appearing in the equality above lie in $K_{\theta\vert_V}(V)$, and thus the above equalities imply, respectively,
	\begin{equation}\label{eq:rel VmodK A}
		1=\prod_{h=0}^{p-1}\left(\left(\overline{z_1^{t_1^k}}\right)^{q-1}\cdot \overline{z_1^{t_1^{k+1}}}\right)  =\left(\left(\overline{z_1^{t_1^k}}\right)^{q-1}\cdot\overline{z_1^{t_1^{k+1}}}\right)^{p}
	\end{equation}
	and
	\begin{equation}\label{eq:rel VmodK B}
		1=\prod_{h=0}^{p-1}\left(\left(\overline{z_1^{t_1^k}}\right)^{-1}\cdot \overline{z_1^{t_1^{k+1}}}\right)  =\left(\left(\overline{z_1^{t_1^k}}\right)^{-1}\cdot\overline{z_1^{t_1^{k+1}}}\right)^{p}
	\end{equation}
	in $\bar V$ --- observe that for $k=p-1$, one has $z_1^{t_1^p}=z_1^v=z_1[z_1,v]$, and thus
	$$\overline{z_1^{t_1^p}}=\bar z_1^{\theta(v)^{-1}}=\bar z_1^{(1-q)^p}\qquad\text{in }\bar V.$$
	Since $\Ker(\theta\vert_V)/K_{\theta\vert_V}(V)$ is an abelian pro-$p$ group, putting together the $p$ relations \eqref{eq:rel VmodK A}, respectively the $p$ relations \eqref{eq:rel VmodK B}, for $0\leq k\leq p-1$, yields, respectively,
	\begin{equation}\label{eq:rel barV A}
		\begin{split}
			1 &= \prod_{k=0}^{p-1}\left(\left(\overline{z_1^{t_1^k}}\right)^{q-1}\cdot\overline{z_1^{t_1^{k+1}}}\right)^{p}\\
			&= \left(\bar z_1^{q-1}\cdot \left(\overline{z_1^{t_1}}\right)^{1+(q-1)}\cdots \left(\overline{z_1^{t_1^{p-1}}}\right)^{1+(q-1)}\cdot\bar z_1^{(1-q)^p}\right)^{p}\\
			&=\left(\bar z_1^{(1-q)\left((1-q)^{p-1}-1\right)}\left(\overline{z_1^{t_1}}\cdots\overline{z_1^{t_1^{p-1}}}\right)^q\right)^p;\end{split}\end{equation}
	and
	\begin{equation}\label{eq:rel barV B}
		\begin{split}
			1 &= \prod_{k=0}^{p-1}\left(\left(\overline{z_1^{t_1^k}}\right)^{-1}\cdot\overline{z_1^{t_1^{k+1}}}\right)^{p}\\
			&= \left(\bar z_1^{-1}\cdot \left(\overline{z_1^{t_1}}\right)^{1-1}\cdots \left(\overline{z_1^{t_1^{p-1}}}\right)^{1-1}\cdot\bar z_1^{(1-q)^p}\right)^{p}=\bar z_1^{p\left((1-q)^p-1\right)},
	\end{split}\end{equation}
	Since we are assuming that $\theta\neq\mathbf{1}_G$, one has $q\neq0$, and hence $$1-q,(1-q)^{p-1}-1,(1-q)^p-1\neq0.$$
	Moreover, the elements $\bar z_1,\overline{z_1^{t_1}},\ldots\overline{z_1^{t_1^{p-1}}}$ are $p$ distinct elements of the minimal generating set $\calX_{\bar V}$.
	Therefore, \eqref{eq:rel barV A}--\eqref{eq:rel barV B} yield (non-trivial) torsion elements of $\bar V$, so that $(V,\theta\vert_V)$ is not Kummerian by Proposition~\ref{prop:K}.
\end{proof}

Altogether, Theorem~\ref{thmA1} follows by Corollary~\ref{cor:G1 no abstorfree} and Proposition~\ref{prop:Vab A}.


\subsection{Massey products}
Finally, we prove Proposition~\ref{prop:G1 massey}--(b).
Given a sequence $\alpha_1,\ldots,\alpha_n$ of elements of $\rmH^1(G)$ of length $n\geq3$ whose associated $n$-fold Massey product $\langle\alpha_1,\ldots,\alpha_n\rangle$ is defined, we may assume that $\alpha_h\neq0$ for all $h=1,\ldots,n$, otherwise --- by Lemma~\ref{lem:masey}--(a) --- there is nothing to prove.
We split the proof of Proposition~\ref{prop:G1 massey}--(b) into three cases.

First, Lemma~\ref{lem:masey}--(b) implies the following more general result.

\begin{prop}
	Let $\alpha_1,\ldots,\alpha_n$ be a sequence of elements of $\rmH^1(G)$ of length $n\geq3$ whose associated $n$-fold Massey product $\langle\alpha_1,\ldots,\alpha_n\rangle$ is defined.
	If $\alpha_h(z_i)\neq0$ for $h\in\{1,n\}$ and $i\in\{1,2\}$, then the $n$-fold Massey product $\langle\alpha_1,\ldots,\alpha_n\rangle$ vanishes.
\end{prop}

\begin{proof}
	We assume that $h=1$ and $i=1$ --- the argument for the other three cases works verbatim, mutatis mutandis.
	Set $\mathcal{S}=\{z_1,z_2,t_1,t_2\}\subseteq\calX_G$. Then
	\[\alpha_1=a_1 z_1^\ast+b_1 t_1^\ast+a_2 z_2^\ast+b_2 t_2^\ast+
	\sum_{s\in\calX_G\smallsetminus\mathcal{S}}\alpha_1(s)s^\ast,\]
	with $a_i=\alpha_1(z_i),b_i=\alpha_1(t_i)$ --- thus, $a_1\neq0$.
	Hence, by \S~\ref{ssec:G1 cohom}, for $0\leq c_1,c_2<p$ one has
	\[ \begin{split}
		\alpha_1\smallsmile (c_1t_1^\ast+c_2z_2^\ast)&= a_1c_1(z_1^\ast\smallsmile t_1^\ast)+a_1c_2(z_1^\ast\smallsmile z_2^\ast)+b_2c_2(t_2^\ast\smallsmile z_2^\ast)\\
		&=(-1)^{\epsilon_1}a_1c_1(x_1^\ast\smallsmile y_1^\ast)++a_1c_2(z_1^\ast\smallsmile z_2^\ast)-(-1)^{\epsilon_2}b_2c_2(x_1^\ast\smallsmile y_1^\ast)\\
		&= \left((-1)^{\epsilon_1}a_1c_1-(-1)^{\epsilon_2}b_2c_2\right)(x_1^\ast\smallsmile y_1^\ast)+
		a_1c_2(z_1^\ast\smallsmile z_2^\ast).
	\end{split}\]
	Therefore, one may obtain any element of $\rmH^2(G)$ as a cup-product $\alpha_1\smallsmile \beta$ with $\beta=c_1t_1^\ast+c_2z_2^\ast$, for a suitable choice of $c_1,c_2$, so that Lemma~\ref{lem:masey}--(b) implies that $\langle\alpha_1,\ldots,\alpha_n\rangle\supseteq\rmH^2(G)$.
\end{proof}

\begin{prop}
	Let $\alpha_1,\ldots,\alpha_n$ be a sequence of elements of $\rmH^1(G)$ of length $n\geq3$ whose associated $n$-fold Massey product $\langle\alpha_1,\ldots,\alpha_n\rangle$ is defined.
	If $\alpha_h(z_i)=0$ for $h=1,\ldots,n$ and $i\in\{1,2\}$, then the $n$-fold Massey product $\langle\alpha_1,\ldots,\alpha_n\rangle$ vanishes.
\end{prop}

\begin{proof}
	If $d=2$, then $G$ is generated by $z_1,z_2,t_1,t_2$ --- namely, $z_1=x_1,t_1=y_1$ and $z_2\in\{x_2,y_2\}$, and 
	\[G=\langle\: z_1,z_2,t_1,t_2\:\mid\:z_1^q[z_1,t_1]\cdot[z_2,t_2]^{\epsilon_2}=[z_1,z_2]=1,\:\epsilon_2\in\{\pm1\}\:\rangle.\]
	Set $D_1=(d_{ij}),D_2=(d_{ij}')\in\dbU_{n+1}$ such that 
	\[d_{ij}=\begin{cases}\alpha_h(t_1)&\text{if }i=h,j=h+1,\\
		0&\text{if }j>i+1,\end{cases}\qquad\text{and}\qquad
	d_{ij}'=\begin{cases}\alpha_h(t_2)&\text{if }i=h,j=h+1,\\
		0&\text{if }j>i+1.\end{cases}\]
	Then obviously
	\[I_{n+1}^q[I_{n+1},D_1]\cdot[I_{n+1},D_2]^{\epsilon_2}=
	[I_{n+1},I_{n+1}]=I_{n+1},\]
	so that the assignment $z_i\mapsto I_{n+1}$ and $t_i\mapsto D_i$, $i=1,2$, yields a homomorphism $G\to\dbU_{n+1}$ satisfying the properties prescribed by Proposition~\ref{prop:massey unip}--(b).
	Therefore, the $n$-fold Massey product $\langle\alpha_1,\ldots,\alpha_n\rangle$ vanishes.
	
	Now suppose $d\geq2$. Let $G_1$ be the Demu\v skin group
	\[G_1=\left\langle\:x_i,y_i:\:i\neq i_1,i_2\:\mid\:x_1^{\epsilon q}\cdot\prod_{i\neq i_1,i_2}[x_i,y_i]=1
	\:\right\rangle,\]
	and $G_2$ the pro-$p$ group
	$$G_2=\langle\: z_1,z_2,t_1,t_2\:\mid\:z_1^{\epsilon'q}[z_1,t_1]\cdot[z_2,t_2]^{\epsilon_2}=[z_1,z_2]=1,\:\epsilon_2\in\{\pm1\}\:\rangle\in\mathcal{F}_2,$$
	where $\epsilon=0$ and $\epsilon'=1$ if $z_1=x_1$, and conversely if $z_1\neq x_1$, and consider the 
	free pro-$p$ product $\bar G=G_1\amalg G_2$.
	Then 
	\[\calX_G=\left(\calX_G\smallsetminus\{z_1,z_2,t_1,t_2\}\right)\:\dot\cup\:\{\:z_1,z_2,t_1,t_2\:\}
	\]
	is a minimal generating set also of $\bar G$, and thus one has an epimorphism $\psi\colon G\to \bar G$, with kernel generated as a normal subgroup of $G$ by
	\[z_1^{\epsilon'q}[z_1,t_1]\cdot[z_2,t_2]^{\epsilon_2}=\left(x_1^{\epsilon q}\cdot\prod_{i\neq i_1,i_2}[x_i,y_i]\right)^{-1}.\]
	Therefore, one has
	\[
	\rmH^1(G_1)\oplus\rmH^1(G_2)\simeq\rmH^1(\bar G)\overset{\sim}{\longrightarrow} \rmH^1(G),
	\]
	where the first isomorphism is due to \cite[Thm.4.1.5]{nsw:cohn}, and the second one is the inflation map $\mathrm{Inf}_{\bar G,G}^1\colon \rmH^1(\bar G)\to\rmH^1(G)$ induced by $\psi$ (cf. \cite{serre:galc}).
	With an abuse of notation, every $\alpha\in\rmH^1(G)$ may be expressed in a unique way as $\alpha=\alpha'+\alpha''$, with $\alpha'\in\rmH^1(G_1)$ and $\alpha''\in\rmH^1(G_2)$.
	In particular, 
	$$\alpha_h=\alpha_h'+\left(\alpha_h(t_1)t_1^\ast+\alpha_h(t_2)t_2^\ast\right)\quad\text{for } h=1,\ldots,n,$$
	and thus $\alpha_h''\smallsmile\alpha_{h+1}''=0$ for all $h=1,\ldots,n-1$.
	Since $\alpha_h\smallsmile\alpha_{h+1}=0$ by Lemma~\ref{lem:masey}--(c), one has $\alpha_h'\smallsmile\alpha'_{h+1}=0$ as well.
	
	Now, the Demu\v skin $G_1$ satisfies the strong $n$-Massey vanishing property (cf., e.g., \cite[Def.~1.2]{pal:Massey} or \cite[Def.~2.4]{cq:massey}), the $n$-fold Massey product $\langle\alpha_1',\ldots,\alpha_n'\rangle$ vanishes in $\bfH^\bullet(G_1)$, and hence Proposition~\ref{prop:massey unip}--(b) yields a homomorphism $\rho_{G_1}\colon G_1\to \dbU_{n+1}$ satisfying $(\rho_{G_1})_{h,h+1}=\alpha'_h$ for all $h=1,\ldots,n-1$.
	On the other hand, the assignment $z_i\mapsto I_{n+1}$ and $t_i\mapsto D_i$, $i=1,2$, as in the case $d=2$, yields a homomorphism $\rho_{G_2}\colon G_2\to\dbU_{n+1}$ satisfying $(\rho_{G_2})_{h,h+1}=\alpha''_h$ for all $h=1,\ldots,n-1$.
	
	Altogether, the universal property of the free pro-$p$ product (cf., e.g., \cite[Ch.~IV, \S~1]{nsw:cohn}) yields a homomorphism $\bar\rho\colon\bar G\to\dbU_{n+1}$, and the composition $\bar\rho\circ\psi\colon G\to\dbU_{n+1}$ satisfies the properties prescribed by Proposition~\ref{prop:massey unip}--(b), proving the claim.
\end{proof}

Finally, we deal with the remaining case for $n=3,4$.

\begin{prop}\label{prop:G1 massey}
	Let $n$ be equal to 3 or 4. If $\alpha_h(z_i)=0$ for $h=1,n$ and $i=1,2$, but $\alpha_h(z_i)\neq0$ for some $2\leq h\leq n-1$ and $i=1,2$, the $n$-fold Massey product $\langle\alpha_1,\ldots,\alpha_n\rangle$ vanishes.
\end{prop}

\begin{proof}
	By Proposition~\ref{prop:massey unip}--(a) there exists a homomorphism $\bar\rho\colon G\to\bar\dbU_n$ such that $\bar\rho_{h,h+1}=\alpha_j$ for all $h=1,\ldots, n$.
	
	Pick $A_1,B_1,\ldots,A_d,B_d\in\dbU_n$ such that 
	$$\bar\rho(x_i)=A_i\cdot\Zen(\dbU_{n+1})\qquad\text{and}\qquad
	\bar\rho(y_i)=B_i\cdot\Zen(\dbU_{n+1})\qquad\text{ for all }i=1,\ldots,d.$$
	Our strategy is to suitably modify (some of) the above matrices in order to construct a homomorphism $\rho\colon G\to\dbU_{n+1}$ satisfying the properties prescribed in Proposition~\ref{prop:massey unip}--(b).
	
	First, in order to simplify the notation, we call the matrices associated with $z_1,z_2$, and with $t_1,t_2$, respectively $C_1,C_2$ and $D_1,D_2$, and set 
	$$C_1=(c_{ij}),\; C_2=(c'_{ij}),\qquad\text{and}\qquad D_1=(d_{ij}),\; D_2=(d'_{ij}).$$
	Without loss of generality, we may suppose that $\alpha_2(z_1)=c_{2,3}\neq0$.
	We claim that 
	\begin{equation}\label{eq:comm C1C2C}
		\left[C_1,C_2\cdot \tilde C\right]=I_{n+1} \qquad \text{or}\qquad
		\left[C_1\cdot \tilde C,C_2\right]=I_{n+1},    
	\end{equation}
	for a suitable $\tilde C=I_{n+1}+\tilde cE_{3,n+1}\in\dbU_5$.
	
	If $n=3$, then an explicit computation shows that $[C_1,C_2]=I_4$, so that $\tilde c=0$ and $\tilde C=I_4$.
	
	If $n=4$, then by
	\eqref{eq:cup 0 massey} one has 
	\[\begin{split}
		0=\alpha_2\smallsmile\alpha_3 &= 
		(z_1^\ast\smallsmile z_2^\ast)\cdot\left(\alpha_2(z_1)\alpha_3(z_2)-\alpha_2(z_2)\alpha_3(z_1)\right)+\\
		&\qquad +(x_1^\ast\smallsmile y_1^\ast)\cdot\sum_{i=1}^d(\alpha_2(x_i)\alpha_3(y_i)-\alpha_2(y_i)\alpha_3(x_i))\\
		&=(c_{2,3}c'_{3,4}-c'_{2,3}c_{3,4})(z_1^\ast\smallsmile z_2^\ast)+c''(x_1^\ast\smallsmile y_1^\ast)
	\end{split}\]
	for some $c''\in\F_p$.
	Since $z_1^\ast\smallsmile z_2^\ast$ and $x_1^\ast\smallsmile y_1^\ast$ are linearly independent elements of $\rmH^2(G)$ (cf. \S\ref{ssec:G1 cohom}), necessarily $c_{2,3}c'_{3,4}=c_{3,4}c'_{2,3}$.
	Therefore $c'_{2,3}=kc_{2,3}$ and $c'_{3,4}=kc_{3,4}$, with $k=c_{2,3}'/c_{2,3}$.
	Then one computes
	\[[C_1,C_2]=\left(\begin{array}{ccccc} 
		1&0& 0 & c_{3,4}(kc_{1,3}-c'_{1,3}) & c_{1,3}c'_{3,5}-c'_{1,3}c_{3,5} \\ 
		&1& 0 & 0 &c_{2,3}(c'_{3,5}-kc_{3,5})\\
		&&1&0 & 0  \\ &&&1&0 \\ &&&& 1\end{array}\right)
	\]
	Since $\bar\rho$ is a morphism from $G$ to $\bar\dbU_5$, and since $[z_1,z_2]=1$, one has $[C_1,C_2]\in\Zen(\dbU_5)$, so that
	\begin{equation}\label{eq:c0 15 C1C2}
		c_{3,4}(kc_{1,3}-c'_{1,3})=c_{2,3}(c'_{3,5}-kc_{3,5})=0.
	\end{equation}
	Consequently, $c'_{3,5}=kc_{3,5}$, as $c_{2,3}\neq0$, and the $(1,5)$-entry of $[C_1,C_2]$ is $c_{3,5}(kc_{1,3}-c'_{1,3})$.
	If also $c_{3,4}\neq0$, then $c'_{1,3}=kc_{1,3}$ by \eqref{eq:c0 15 C1C2}, and $[C_1,C_2]=I_5$.
	If instead $c_{3,4}=0$ and $[C_1,C_2]\neq I_5$ --- namely, $c'_{1,3}\neq kc_{1,3}$ ---, then $c_{1,3},c'_{1,3}$ are not both 0.
	\begin{itemize}
		\item[(a)] If $c_{1,3}\neq0$, then set $\tilde c=-c_{3,5}(kc_{1,3}-c'_{1,3})$.
		Then 
		\[\begin{split}
			\left[C_1,C_2\tilde C\right]&= [C_1,\tilde C]\cdot[C_1,C_2]^{\tilde C}\\
			&= [C_1,C_2][C_1,\tilde C]\\&=I_5+(c_{3,5}(kc_{1,3}-c'_{1,3})-c_{3,5}(kc_{1,3}-c'_{1,3}))E_{1,5}=I_5.
		\end{split}\]    
		\item[(b)] If $c'_{1,3}\neq0$, then set $\tilde c=c_{3,5}(kc_{1,3}-c'_{1,3})$.
		Then 
		\[\begin{split}
			\left[C_1\tilde C,C_2\right]&= [C_1,C_2]^{\tilde C}\cdot[\tilde C,C_2]\\
			&= [C_1,C_2][\tilde C,C_2]\\&=I_5+(c_{3,5}(kc_{1,3}-c'_{1,3})-c_{3,5}(kc_{1,3}-c'_{1,3}))E_{1,5}=I_5.
		\end{split}\]
	\end{itemize}
	This shows \eqref{eq:comm C1C2C}.
	
	Now, since $\bar\rho$ is a morhpism from $G$ to $\bar\dbU_{n+1}$, one has 
	\begin{equation}\label{eq:rel demushkin mod ZU5 A}
		A_1^q[A_1,B_1]\cdots[A_d,B_d] \in\Zen(\dbU_{n+1})    
	\end{equation}
	--- recall that $A_{i_1}=C_1$ and $B_{i_1}=D_1$, if $z_1=x_{i_1}$; or $B_{i_1}=C_{1}$ and $A_{i_1}=D_1$, if $z_1=y_{i_1}$; and analogously with $z_2$ and $C_2,D_2$.
	If $[C_1,C_2]\neq I_{n+1}$, then $n=4$, $\tilde C=I_5+\tilde cE_{3,5}$, and for any matrix $D=(u_{ij})\in\dbU_5$ one has
	\[[D,\tilde C]=\left(\begin{array}{ccccc} 1 &0&0&0& \tilde u \\
		&1&0&0& u_{2,3}\tilde c \\ &&1&0&0 \\&&&1&0 \\ &&&&1\end{array}\right)\]
	for some $\tilde u\in\F_p$.
	Therefore, if one replaces $C_2$ with $C_2\tilde C$, or $C_1$ with $C_1\tilde C$ --- depending on the two cases above --- in \eqref{eq:rel demushkin mod ZU5 A} --- and hence $D=D_1$ or $D=D_2$, again depending on the two cases above ---, one obtains
	\[
	A_1^q[A_1,B_1]\cdots[A_d,B_d]\equiv [D',\tilde C]^\epsilon\equiv I_5+\epsilon u_{2,3}\tilde cE_{2,5}\mod \Zen(\dbU_{n+1})
	\]
	for some $\epsilon\in\{\pm1\}$ depending on $\epsilon_1,\epsilon_2$ --- observe that $[D,\tilde C]$ commutes with any matrix of $\dbU_5$ modulo $\Zen(\dbU_5)$, by Fact~\ref{fact:matrix}.
	Also, we remark that, in case $i_1=1$ --- and hence necessarily $z_1=x_1$ (cf. Proposition~\ref{prop:G1 only orient}) --- one has $(C_1\tilde C)^q\equiv C_1^q$ modulo $\Zen(\dbU_5)$, as $[C_1,\tilde C]^{\binom{q}{2}}=I_5$ and $[C_1,[C_1,\tilde C]]\in\Zen(\dbU_5)$.
	Therefore, we need to suitably modify one of the matrices $D_1,D_2$ in order to ``restore'' \eqref{eq:rel demushkin mod ZU5 A}.
	\begin{itemize}
		\item[(a)] If $c_{1,3}\neq0$ then the $(2,5)$-entry of $[\tilde C,D_2]^{\epsilon_2}$ is $-\epsilon_2 d'_{2,3}\tilde c$. Set $$\tilde D=I_5+\frac{\epsilon_2d'_{2,3}\tilde c}{c_{2,3}}E_{3,5}.$$ Then $[C_1,\tilde D]\equiv I_5+\epsilon_2d'_{2,3}\tilde cE_{2,5}$ modulo $\Zen(\dbU_5)$, and hence
		\[\begin{split}
			[C_1,\tilde DD_1]^{\epsilon_1}&=\left([C_1,D_1][C_1,\tilde D]^{D_1}\right)^{\epsilon_1}\\
			&\equiv [C_1,D_1]^{\epsilon_1}\cdot[\tilde C,D_2]^{-\epsilon_2}\mod \Zen(\dbU_5),    
		\end{split}\]
		and after replacing $C_2$ with $C_2\tilde C$ and $D_1$ with $\tilde D D_1$ in \eqref{eq:rel demushkin mod ZU5 A}, one obtains again a matrix in $\Zen(\dbU_5)$.
		
		\item[(b)] Similarly, if $c'_{1,3}\neq0$ then the $(2,5)$-entry of $[\tilde C,D_1]^{\epsilon_1}$ is $-\epsilon_1 d_{2,3}\tilde c$. Set $$\tilde D=I_5+\frac{\epsilon_2d_{2,3}\tilde c}{c'_{2,3}}E_{3,5}.$$ Then $[C_2,\tilde D]\equiv I_5+\epsilon_2d_{2,3}\tilde cE_{2,5}$ modulo $\Zen(\dbU_5)$, and hence
		\[\begin{split}
			[C_2,\tilde DD_2]^{\epsilon_2}&=\left([C_2,D_2][C_2,\tilde D]^{D_2}\right)^{\epsilon_2}\\
			&\equiv [C_2,D_2]^{\epsilon_2}\cdot[\tilde C,D_1]^{-\epsilon_1}\mod \Zen(\dbU_5),    
		\end{split}\]
		and after replacing $C_1$ with $C_1\tilde C$ and $D_2$ with $\tilde D D_2$ in \eqref{eq:rel demushkin mod ZU5 A}, one obtains again a matrix in $\Zen(\dbU_5)$.
	\end{itemize}
	
	Altogether, we may assume that --- for both $n=3,4$ --- one has matrices $A_1,\ldots,B_d$ such that \eqref{eq:rel demushkin mod ZU5 A} holds, and moreover $[C_1,C_2]=I_{n+1}$. 
	Hence, by \eqref{eq:rel demushkin mod ZU5 A} one has
	\[A_1^q[A_1,B_1]\cdots[A_d,B_d]=I_{n+1}+\tilde aE_{1,n+1} \in\Zen(\dbU_{n+1})   \]
	for some $\tilde a\in\F_p$.
	To conclude the proof, we need to suitably modify one of the matrices $A_1,\ldots,B_d$.
	Since we are assuming that $\alpha_1\neq0$, one has $\alpha_1(x_j)\neq0$ or $\alpha_1(y_j)\neq0$ for some $j\in\{1,\ldots,d\}$.
	\begin{itemize}
		\item[(a)] If $\alpha_1(x_j)\neq 0$ then set $\tilde B=I_{n+1}-\tilde a/\alpha_1(x_j)E_{2,n+1}$. 
		Then 
		\[    \left[A_j,B_j\tilde B\right] =[A_j,\tilde B]\cdot[A_j,B_j]^{\tilde B}=\left(I_{n+1}-\tilde a E_{1,n+1}\right)\cdot[A_j,B_j],       \]
		as $\tilde B$ and $[A_j,B_j]$ commute by Fact~\ref{fact:matrix}.
		Observe that if $x_j=t_1$ --- respectively, $x_j=t_2$ ---, then replacing $B_j=C_1$ with $C_1\tilde B$ yields
		$$[C_1\tilde B,C_2]=[C_1,C_2]^{\tilde B}\cdot [\tilde B,C_2]=I_{n+1}\cdot I_{n+1},$$
		as $[\tilde B,C_2]=I_{n+1}$, for the $(1,2)$-entry and the $(n,n+1)$-entry of $C_2$ are 0 --- respectively (and analogously), replacing $B_j=C_2$ with $C_2\tilde B$, yields $[C_1,C_2\tilde B]=I_{n+1}$.
		\item[(b)] Analogously, if $\alpha_1(y_j)\neq 0$ then set 
		$\tilde A=I_{n+1}+\tilde a/\alpha_1(y_j)E_{2,n+1}$.
		Then 
		\[  \left[A_j\tilde A,B_j\right] =[A_j,B_j]^{\tilde A}\cdot[\tilde A,B_j]=[A_j,B_j]\cdot \left(I_{n+1}-\tilde a E_{1,n+1}\right),   \]
		as $\tilde A$ and $[A_j,B_j]$ commute by Fact~\ref{fact:matrix}.
		As above, we remark that if $y_j=t_1$ --- respectively, $y_j=t_2$ --- then replacing $A_j=C_1$ with $C_1\tilde A$ yields $[C_1\tilde A,C_2]=I_{n+1}$ --- respectively, replacing $A_j=C_2$ with $C_2\tilde A$ yields $[C_1,C_2\tilde A]=I_{n+1}$ ---, with the same argument as above.
	\end{itemize}
	
	Therefore, if one replaces $B_j$ with $B_j\tilde B$, or $A_j$ with $A_j\tilde A$, depending on the two cases above, one obtains
	\[A_1^q[A_1,B_1]\cdots[A_d,B_d]=[C_1,C_2]=I_{n+1}\]
	--- once again, we remark that if $\alpha_1(y_1)\neq 0$, then one has $(A_1\tilde A)^q=A_1^q$, because of $[A_1,\tilde A]^{\binom{q}{2}}=I_{n+1}$ and $[A_1,\tilde A]\in\Zen(\dbU_5)$.
	Thus, this assignment yields a homomorphism $\rho\colon G\to \dbU_{n+1}$ satisfying the properties prescribed in Proposition~\ref{prop:massey unip}--(b), and the $n$-fold Massey product $\langle\alpha_1,\ldots,\alpha_n\rangle$ vanishes.
\end{proof}

We suspect that such a pro-$p$ group $G$ satisfies, in fact, the $n$-fold Massey vanishing property for every $n\geq3$.

\begin{question}
	Let $G$ be a pro-$p$ group lying in the family $\mathcal F_2$. Does $G$ satisfy the $n$-fold Massey vanishing property for every $n\geq3$?
\end{question}

\subsection{The graded group algebra}
Consider the graded group algebra $\gr \FppG$, generated by $X_1,Y_2\ldots,X_d,Y_d$.
The defining relations of $G$ have initial form the Lie polynomials
\[
(X_1Y_1-Y_1X_1)+\ldots+(X_dY_d-Y_dX_d)\qquad\text{and}\qquad Z_1Z_2-Z_2Z_1,
\]
with $Z_1,Z_2$ as in \S~\ref{ssec:intro group algebras}.
Since $G$ is a mild pro-$p$ group by Proposition~\ref{prop:mild},
\cite[Thm.~2.12--(3)]{Jochen} implies that
\[
    \grFp G\simeq \frac{T^\bullet(X_1,Y_1,\dots,X_d,Y_d)}{([X_1,Y_1]+\dots+[X_d,Y_d],[Z_1,Z_2])},
\]
and moreover, \cite[Thm.~1.3]{expo} implies Proposition~\ref{prop:properties intro}--(3) for $\mathcal F_2$.

Moreover, by Jennings' theorem and Theorem~\ref{thm:2rel}, the restricted Lie algebra $\gr G$ is Bloch-Kato, with cohomology $\ext^\bullet_{u(\gr G)}(\F_p,\F_p)\simeq \bfH^\bullet(G)$. 
This proves Proposition~\ref{prop:gradedalgebras intro} for $\mathcal F_2$.

To boot, we have the following result on the restricted Lie algebras associated to the pro-$p$ groups lying in either of the two families $\mathcal{F}_1,\mathcal{F}_2$.

 \begin{prop}\label{prop:free}
    Let $G$ be a pro-$p$ group lying in one of the two families $\mathcal F_i$ ($i=1,2$), and let $\mathfrak g$ be the ordinary subalgebra of $\gr G$ generated by $\gr _1G=G/G^p[G,G]$. 
    Then, the derived subalgebra $\mathfrak g':=[\mathfrak g,\mathfrak g]$ is a free Lie algebra. In particular, all finitely generated subalgebras of $\mathfrak g$ are of type FP$_\infty$, namely, they have finite trivial-coefficients cohomology groups.
 \end{prop}
 \begin{proof}
    If $G$ belongs to the class $\mathcal F_1$, then $\mathfrak g$ is the free product of a surface Lie algebra and a free Lie algebra of rank $2$. Since all the proper subalgebras of a surface Lie algebra are free, it follows from \cite[Prop. 2.10]{bassSerreLie} that $\mathfrak g'$ is a free Lie algebra.

    Now let $G$ be a group of the family $\mathcal F_2$.
     Up to isomorphism, $\mathfrak g$ admits the following presentation as an ordinary Lie algebra: \[\pres{X_1,Y_1,\dots,X_d,Y_d}{\sum_{i=1}^d[X_i,Y_i],[Y_1,Y_2]}.\] 
     Let $\phi:\gen{Y_1}\to\mathfrak h=\gen{Y_1,X_2,Y_2,\dots,X_d,Y_d}$ be the derivation defined by $\phi(Y_1)=\sum_{i=2}^d [Y_i,X_i]$. Since $\mathfrak g$ is Bloch-Kato, the subalgebra $\mathfrak h$ generated by $Y_1,X_2,Y_2,\dots,X_d,Y_d$ is quadratic, and hence it has presentation $\pres{Y_1,X_2,Y_2,\dots,X_d,Y_d}{[Y_1,Y_2]}$. 
     In particular, $\mathfrak h$ has free derived subalgebra. 
     Now, by \cite{hnnLS}, since $\gen {Y_1}\cap \mathfrak h'=0$, we deduce that $\mathfrak g'$ is a free Lie algebra.

     By \cite[Prop. 5.8]{bassSerreLie}, $\mathfrak g$ is locally of type FP, that is, $\ext^n_{U(\mathfrak g)}(\F_p,\F_p)$ is finite. 
 \end{proof}

 The finitely generated closed subgroups of a Bloch-Kato pro-$p$ group have quadratic cohomology, and, in particular, they admit a finite presentation. This means that Bloch-Kato pro-$p$ groups are coherent. 

 Although it is not known whether Bloch-Kato restricted Lie algebras are coherent, we only know that, by definition, their subalgebras generated in degree one are of type FP$_\infty$. 
 From Proposition \ref{prop:free} we deduce that the \textit{ordinary} Lie subalgebra of $\gr G$ generated in degree $1$ is coherent, when $G$ belongs to one of the two families $\mathcal F_i$'s.

 \begin{question}
     Let $\mathfrak g$ be a Bloch-Kato restricted Lie algebra, and let $\mathfrak h$ be a finitely generated restricted subalgebra. Is $\mathfrak h$ finitely presented?
 \end{question}


\section{Concerning the last example}
We conclude with some properties of the pro-$p$ group with presentation \eqref{eq:pres question} and its associated graded algebras.

First, we observe that by \cite[Thm.~B]{qsv:quadratic} the $\F_p$-cohomology algebra $\bfH^\bullet(G)$ is quadratic, and $\cd(G)=2$.

Now, the initial forms of the defining relations of $G$ in $\mathrm{gr}\:G$ become
\[X_1Y_1-Y_1X_1=X_2Y_2-Y_2X_2=\ldots=X_dY_d-Y_dX_d.
\]
Considering the lexicographic order, the leading monomials of these relations are $Y_dX_d$, $Y_{d-1}X_{d-1}$, $\ldots$, $Y_2X_2$, and therefore by \cite[Thm.~3.5]{Jochen} $G$ is a mild pro-$p$ group.
Consequently, the graded restricted Lie algebra associated to $G$ is 
$$\gr \:G=\pres{X_1,Y_1,X_2,Y_2,X_3,Y_3}{[X_1,Y_1]=[X_2,Y_2]=[X_3,Y_3]}.$$

By Theorem~\ref{thm:2rel}, $\gr G$ is Bloch-Kato, with cohomology 
$$\ext^\bullet_{u(\gr G)}(\F_p,\F_p)=\bfH^\bullet(G)=\Lambda^\bullet(x_i^\ast,y_i^\ast)/(\Omega),$$ where $\Omega$ is the subspace of $\Lambda^2(x_i^\ast,y_i^\ast)$ generated by the elements $x_i^\ast\wedge y_j^\ast,\ x_i^\ast\wedge x_j^\ast,\ y_i^\ast\wedge y_j^\ast$ ($i\neq j$), and $x_1^\ast\wedge y_1^\ast+x_2^\ast\wedge y_2^\ast+x_3^\ast\wedge y_3^\ast$. In particular, this cohomology algebra is universally Koszul.

In \cite[\S~4.3]{simone:koszul}, the first named author studies this restricted Lie algebra as an ordinary Lie algebra.
\begin{prop}\label{prop:b2d}
    Let $d\geq 1$ be an integer, and consider the ordinary Lie $\F_p$-algebra 
    $$\mathfrak b(2d)=\pres{X_1,Y_1,\dots,X_n,Y_n}{[X_i,Y_i]-[X_1,Y_1]:\ 2\leq i\leq d}. $$

    \begin{enumerate}
        \item The Lie algebra $\mathfrak b(2d)$ is Bloch-Kato;
        \item The derived subalgebra $\mathfrak b(2d)'$ is free, and all its finitely generated subalgebras are of type FP;
        \item The cohomology algebra $\bfH^\bullet(\mathfrak b(2d),\F_p)$ is universally Koszul, and admits a presentation $\Lambda^\bullet(X_i^\ast,Y_i^\ast)/(\Omega)$, where $\Omega$ is spanned by the elements $X_i^\ast\wedge Y_j^\ast,\ X_i^\ast\wedge X_j^\ast,\ Y_i^\ast\wedge Y_j^\ast$ ($i\neq j$), and $\sum_{i=1}^d X_i^\ast\wedge Y_i^\ast$;
        \item The cohomological dimension of $\mathfrak b(2d)$ is two.
    \end{enumerate}
\end{prop}

Finally, it is not difficult to see that $G$ satisfies the 3-fold Massey vanishing property. 
Indeed, let $\alpha_1,\alpha_2,\alpha_3$ be a sequence of elements of $\rmH^1(G)$, and let $\bar \rho\colon G\to\bar\dbU_4$ be a homomorphism satisfying the properties prescribed in Proposition~\ref{prop:massey unip}--(a). Put $A_i=\bar\rho(x_i)$ and $B_i=\bar\rho(y_i)$ for $i=1,\ldots,d$.
Then 
\[[A_1,B_1]\equiv[A_2,B_2]\equiv\ldots\equiv[A_d,B_d]\mod\Zen(\dbU_4).\]
Pick $C\in\dbU_4$ such that $C\equiv[A_1,B_1]\bmod \Zen(\dbU_4)$ and the (1,4)-entry of $C$ is 0.
If $\alpha_h(x_i)=\alpha_h(y_i)=0$ for both $h=1,3$ and for some $i$, then $[A_i,B_i]=I_4$, and hence necessarily $C=I_4$.
On the other hand, if $\alpha_h(x_i)\neq0$ or $\alpha_h(y_i)\neq0$ with $h\in\{1,3\}$, then we proceed as in the proof of Proposition~\ref{prop:G1 massey}: if $\alpha_h(x_i)\neq0$, then we may find a suitable matrix $\tilde B=I_4+b'E_{1,3}+b''E_{2,4}$ such that $[A_i,B_i\tilde B]=C$; while if $\alpha_h(y_i)\neq0$ then we may find a suitable matrix $\tilde A=I_4+a'E_{1,3}+a''E_{2,4}$ such that $[A_i\tilde A,B_i]=C$ --- we leave the details to the reader.
After replacing $A_i$ or $B_i$ accordingly, the assignment $x_i\mapsto A_i$ and $y_i\mapsto B_i$ yields a homomorphism $\rho\colon G\to\dbU_4$ as prescribed in Proposition~\ref{prop:massey unip}--(b).

Whether $G$ is a Bloch-Kato (or 1-cyclotomic) pro-$p$ group is still unknown to the authors. 
As remarked for the pro-$p$ groups in the families $\mathcal{F}_1,\mathcal{F}_2$, since $\cd(G)=2$, in order to prove the Bloch-Kato property it suffices to check the shape of the defining relations of the open subgroups of $G$.

\end{document}